\documentclass[12pt,reqno,a4paper]{amsart}            

\usepackage[totalwidth=450pt, totalheight=650pt]{geometry}

\usepackage{ifthen}       
\usepackage{scrtime}     
\newcommand{\olafsMacrosIncluded}{} 

\usepackage{scrpage2}    
\usepackage{scrtime}     

\ifthenelse{\isundefined \OlafsVersion}
{ 
}
{ 
}

\ifthenelse{\isundefined \draft}
{ 
  \lohead{\textsf{\shorttitle}}
  \rehead{\textsf{\authors}}
  \automark[subsection]{section}

  \newcommand{\look}[1]{}%
  \newcommand{\lookO}[1]{}%
  \newcommand{\lookR}[1]{}%
}
{ 
  \rehead{\texttt{\jobname.tex}, \today, \thistime}
  \lohead{\texttt{\jobname.tex}, \today, \thistime}

  \newcommand{\marker}[1]{\fbox{\rule{0pt}{0.1ex}\textbf{#1}}}
  \newcommand{\markerO}{\marker{Olaf}}
  \newcommand{\markerR}{\marker{Ralf}}

  \newcounter{lookcounter}
  \setcounter{lookcounter}{0}
  \newcommand{\look}[1]
             {
               \stepcounter{lookcounter} 
               \marker{$\diamond$}
               \footnote[\arabic{lookcounter}]{\marker{$\diamond$} #1}
             }
  \newcommand{\lookO}[1]
             {
               \stepcounter{lookcounter} 
               \markerO
               \footnote[\arabic{lookcounter}]{\markerO #1}
             }
  \newcommand{\lookR}[1]
             {
               \stepcounter{lookcounter} 
               \markerR
               \footnote[\arabic{lookcounter}]{\markerR #1}
             }

  \ifthenelse{\isundefined \OlafsVersion} 
  {\usepackage[notref,notcite]{showkeys}}              
  {\usepackage{/home/post/Aktuell/LaTeX/my-showkeys}}  
}

\newcommand{\comment}[1]{}   

\usepackage{graphicx}                

\usepackage{amsmath, amssymb}
\usepackage{amsthm}

\usepackage{amscd}

\usepackage{accents}     
\usepackage{yhmath}     

\usepackage{mathtools}  

\usepackage{bbm}         

\usepackage{mathrsfs}    
\renewcommand \mathcal \mathscr

\numberwithin{equation}{section}

\newcounter{myenumi}

\newcommand{\itemref}[1]{(\ref{#1})}

\newcommand{\myfont}{\sffamily}

\newtheoremstyle{mythmstyle}
  {\topsep}
  {\topsep}
  {\itshape}
  {}
  {\bfseries \myfont}
  {.}
  {.5em}
  {}

\newtheoremstyle{mydefstyle}
  {\topsep}
  {\topsep}
  {\normalfont}
  {}
  {\bfseries \myfont}
  {.}
  {.5em}
  {}

\swapnumbers                    

\theoremstyle{mythmstyle}       

\newtheorem{theorem}{Theorem}[section]
\newtheorem{maintheorem}[theorem]{Main Theorem}
\newtheorem{proposition}[theorem]{Proposition}
\newtheorem{lemma}[theorem]{Lemma}
\newtheorem{corollary}[theorem]{Corollary}

\newcounter{intro}

\newtheorem{conjecture}[theorem]{Conjecture}

\theoremstyle{mydefstyle}        
\newtheorem{definition}[theorem]{Definition}
\newtheorem{assumption}[theorem]{Assumption}
\newtheorem{example}[theorem]{Example}
\newtheorem{examples}[theorem]{Examples}

\newtheorem{remark}[theorem]{Remark}
\newtheorem{remarks}[theorem]{Remarks}
\newtheorem*{remark*}{Remark}
\newtheorem{notation}[theorem]{Notation}

\expandafter\let\expandafter\oldproof\csname\string\proof\endcsname
\let\oldendproof\endproof
\renewenvironment{proof}[1][\bfseries\myfont\proofname]{%
  \oldproof[\bfseries \myfont #1]%
}{\oldendproof}

\makeatletter
\renewcommand\section{\@startsection{section}{1}%
  \z@{.7\linespacing\@plus\linespacing}{.5\linespacing}%
  {\Large\myfont\bfseries}}
\renewcommand\subsection{\@startsection{subsection}{2}%
  \z@{-.5\linespacing\@plus-.7\linespacing}{.5\linespacing}%
  {\large\myfont\bfseries}}
\renewcommand\subsubsection{\@startsection{subsubsection}{3}%
  \z@{.5\linespacing\@plus.7\linespacing}{-.5em}%
  {\myfont\bfseries}}
\renewenvironment{abstract}{%
  \ifx\maketitle\relax
    \ClassWarning{\@classname}{Abstract should precede
      \protect\maketitle\space in AMS document classes; reported}%
  \fi
  \global\setbox\abstractbox=\vtop \bgroup
    \normalfont\Small
    \list{}{\labelwidth\z@
      \leftmargin3pc \rightmargin\leftmargin
      \listparindent\normalparindent \itemindent\z@
      \parsep\z@ \@plus\p@
      
    }%
    \item[\hskip\labelsep
      \myfont\bfseries
    \abstractname.]%
}{%
  \endlist\egroup
  \ifx\@setabstract\relax \@setabstracta \fi
}
\renewcommand\contentsnamefont{\myfont\bfseries}
\renewcommand\@starttoc[2]{\begingroup
  \setTrue{#1}%
  \par\removelastskip\vskip\z@skip
  \@startsection{}\@M\z@{\linespacing\@plus\linespacing}%
    {.5\linespacing}{
      \contentsnamefont}{#2}%
  \ifx\contentsname#2%
  \else \addcontentsline{toc}{section}{#2}\fi
  \makeatletter
  \@input{\jobname.#1}%
  \if@filesw
    \@xp\newwrite\csname tf@#1\endcsname
    \immediate\@xp\openout\csname tf@#1\endcsname \jobname.#1\relax
  \fi
  \global\@nobreakfalse \endgroup
  \addvspace{32\p@\@plus14\p@}%
  \let\tableofcontents\relax
}
\renewcommand\@settitle{\begin{center}%
  \baselineskip14\p@\relax
    \LARGE
    \bfseries
    \myfont
  \@title
  \end{center}%
}
\renewcommand\@setauthors{%
  \begingroup
  \def\thanks{\protect\thanks@warning}%
  \trivlist
  \centering\footnotesize \@topsep30\p@\relax
  \advance\@topsep by -\baselineskip
  \item\relax
  \author@andify\authors
  \def\\{\protect\linebreak}%
  \large
  \myfont\bfseries\authors
  \ifx\@empty\contribs
  \else
    ,\penalty-3 \space \@setcontribs
    \@closetoccontribs
  \fi
  \endtrivlist
  \normalfont\myfont\@setaddresses
  \endgroup
}
\renewcommand\@setaddresses{\par
  \nobreak \begingroup
\footnotesize
  \def\author##1{\nobreak\addvspace\bigskipamount}%
  \def\\{\unskip, \ignorespaces}%
  \interlinepenalty\@M
  \def\address##1##2{\begingroup
    \par\addvspace\bigskipamount\indent
    \@ifnotempty{##1}{(\ignorespaces##1\unskip) }%
    {
      \ignorespaces##2}\par\endgroup}%
  \def\curraddr##1##2{\begingroup
    \@ifnotempty{##2}{\nobreak\indent\curraddrname
      \@ifnotempty{##1}{, \ignorespaces##1\unskip}\/:\space
      ##2\par}\endgroup}%
  \def\email##1##2{\begingroup
    \@ifnotempty{##2}{\nobreak\indent\emailaddrname
      \@ifnotempty{##1}{, \ignorespaces##1\unskip}\/:\space
      \ttfamily##2\par}\endgroup}%
  \def\urladdr##1##2{\begingroup
    \def~{\char`\~}%
    \@ifnotempty{##2}{\nobreak\indent\urladdrname
      \@ifnotempty{##1}{, \ignorespaces##1\unskip}\/:\space
      \ttfamily##2\par}\endgroup}%
  \addresses
  \endgroup
}
\renewcommand\enddoc@text{\ifx\@empty\@translators \else\@settranslators\fi
}
\renewcommand\@secnumfont{\myfont\bfseries} 

\renewcommand\maketitle{\par
  \@topnum\z@ 
  \@setcopyright
  \thispagestyle{firstpage}
  \ifx\@empty\shortauthors \let\shortauthors\shorttitle
  \else \andify\shortauthors
  \fi
  \@maketitle@hook
  \begingroup
  \@maketitle
  \toks@\@xp{\shortauthors}\@temptokena\@xp{\shorttitle}%
  \toks4{\def\\{ \ignorespaces}}
  \edef\@tempa{%
    \@nx\markboth{\the\toks4
      \@nx\MakeUppercase{\the\toks@}}{\the\@temptokena}}%
  \@tempa
  \endgroup
  \c@footnote\z@
  \@cleartopmattertags
}
\makeatother

\newcommand{\Sec}[1]{Section~\ref{sec:#1}}

\newcommand{\Subsec}[1]{Subsection~\ref{ssec:#1}}

\newcommand{\Thm}[1]{Theorem~\ref{thm:#1}}

\newcommand{\Thmenum}[2]{Theorem~\ref{thm:#1}~(\ref{#2})}

\newcommand{\Mthm}[1]{Main Theorem~\ref{mthm:#1}}
\newcommand{\Mthms}[2]{Main Theorems~\ref{mthm:#1} and~\ref{mthm:#2}}

\newcommand{\Exenum}[2]{Example~\ref{ex:#1}~(\ref{#2})}

\newcommand{\Lem}[1]{Lemma~\ref{lem:#1}}

\newcommand{\Cor}[1]{Corollary~\ref{cor:#1}}

\newcommand{\Prp}[1]{Proposition~\ref{prp:#1}}

\newcommand{\Rem}[1]{Remark~\ref{rem:#1}}

\newcommand{\Def}[1]{Definition~\ref{def:#1}}

\newcommand{\Defenum}[2]{Definition~\ref{def:#1}~(\ref{#2})}

\newcommand{\Ass}[1]{Assumption~\ref{ass:#1}}

\newcommand{\abs}[2][{}]{\lvert{#2}\rvert_{{#1}}}    
\newcommand{\abssqr}[2][{}]{\lvert{#2}\rvert^2_{#1}} 
\newcommand{\Bigabs}[2][{}]{\Bigl\lvert{#2}\Bigr\rvert_{#1}}     

\newcommand{\normsymb}{\|}

\newcommand{\norm}[2][{}]{\normsymb{#2}\normsymb_{{#1}}}    
\newcommand{\normsqr}[2][{}]{\normsymb{#2}\normsymb^2_{#1}} 

\newcommand{\iprod}[3][{}]{\langle{#2},{#3}\rangle_{#1}}  

\newcommand{\set}[2]{\{ \, #1 \, | \, #2 \, \} }      
\newcommand{\bigset}[2]{\bigl\{ \, #1 \, \bigl|\bigr. \, #2 \, \bigr\} }
\newcommand{\Bigset}[2]{\Bigl\{ \, #1 \, \Bigl|\Bigr. \, #2 \, \Bigr\} }

\DeclareMathOperator*{\bigdcup}{\mathaccent\cdot{\bigcup}}

\newcommand{\map}[3]{ #1 \colon #2 \longrightarrow #3}    

\newcommand{\bd}  {\partial}          
\newcommand{\clo}[2][]{\overline{{#2}}^{#1}} 

\newcommand{\restr}[1]{{\restriction}_{#1}} 

\def\XXint#1#2#3{{\setbox0=\hbox{$#1{#2#3}{\int}$}
     \vcenter{\hbox{$#2#3$}}\kern-.5\wd0}}
\def\XXsum#1#2#3{{\setbox0=\hbox{$#1{#2#3}{\int}$}
     \vcenter{\hbox{$#2#3$}}\kern-.60\wd0}}

\DeclareMathOperator{\dd}    {d\!}

\DeclareMathOperator{\dom}    {dom}

\DeclareMathOperator{\supp}   {supp}
\DeclareMathOperator{\vol}    {vol}

\newcommand{\eps}{\varepsilon} 
\renewcommand{\phi}{\varphi}   
\renewcommand{\rho}{\varrho}   

\newcommand{\R}{\mathbb{R}} 
\newcommand{\N}{\mathbb{N}} 

\newcommand{\e}{\mathrm e}  

\newcommand{\wt}{\widetilde}           
\newcommand {\qf}[1]{\mathcal{#1}}    

\newcommand{\Sobsymb} {\mathsf H} 

\newcommand{\Contsymb} {\mathsf C}     
\newcommand{\Lsymb}    {\mathsf L}     

\newcommand{\Sobspace}[1][1]{\Sobsymb^{#1}} 

\newcommand{\Contspace}[1][{}]{\Contsymb^{#1}}     
\newcommand{\Lpspace}[1][p]    {\Lsymb_{#1}}     
\newcommand{\Lsqrspace}    {\Lpspace[2]}     

\newcommand{\Cont}[2][{}]{\Contspace[#1]({#2})}

\newcommand{\Contc}[2][{}]{\Contspace[#1]_{\mathrm c}({#2})}

\newcommand{\Lsqr}[2][{}]{\Lsqrspace^{#1}({#2})} 
\newcommand{\Neu}{\mathrm {Neu}}              
\newcommand{\Dir}{\mathrm {Dir}}              
\newcommand{\laplacian}{\Delta}
\newcommand{\laplacianD}{\laplacian^{\!\Dir}} 
\newcommand{\laplacianN}{\laplacian^{\!\Neu}} 

\usepackage[toc,page]{appendix}
\newcommand {\Lip}{\mathrm{Lip}}  
\newcommand{\LipCont}[1]{\Contspace_\Lip({#1})}
\newcommand {\Proba}{\mathbb P}  
\newcommand{\normvec}{\mathrm{n}}  
\newcommand{\normder}{\partial_\normvec}  
\renewcommand{\bar}[1]{\tilde{#1}}  
\newcommand{\dbar}[1]{\hat{#1}}  
\newcommand{\oS}[1]{#1{}'}   
\usepackage[utf8x]{inputenc}     
\usepackage[british]{babel}      

\usepackage{amsmath}  
\usepackage{amsfonts} 
\usepackage{amssymb}  
\usepackage{amsthm}   

\usepackage{graphicx} 

\usepackage{comment} 

\usepackage{ae}      

\ifthenelse{\isundefined \olafsMacrosIncluded}
  {
    \newcommand{\N}{\mathbb{N}}

    \newcommand{\R}{\mathbb{R}}
    
  }
  {
  }

\renewcommand{\le}{\leqslant}
\renewcommand{\ge}{\geqslant}

\ifthenelse{\isundefined \olafsMacrosIncluded}
  {
    \DeclareMathOperator{\vol}{vol}  
    \DeclareMathOperator{\dom}{dom}
  }
  {
  } 

\newcommand{\ra}{\rightarrow}

\usepackage{enumerate}    

\newenvironment{enumiii}{\begin{enumerate}[(i)]}{\end{enumerate}}
\newenvironment{enumABC}{\begin{enumerate}[(A)]}{\end{enumerate}}

\ifthenelse{\isundefined \olafsMacrosIncluded}
  {
    \newtheoremstyle{example}
    {\topsep}{\topsep} 
    {\itshape}
    {}
    {\bfseries}
    {.}
    {\newline}
    {\thmname{#1}\thmnumber{ #2}\thmnote{ #3}}

    \theoremstyle{example} 
    
    \newtheorem{theorem}{Theorem}[section] 
    \newtheorem{proposition}[theorem]{Proposition} 
    \newtheorem{lemma}[theorem]{Lemma}
    \newtheorem{corollary}[theorem]{Corollary}
    \newtheorem{definition}[theorem]{Definition}
    
    \newtheorem{examples}[theorem]{Examples}
    \newtheorem{maintheorem}[theorem]{Main Theorem}
    \newtheorem{assumption}[theorem]{Assumption}

    \theoremstyle{remark} \newtheorem{remark}[theorem]{Remark}

    \newtheorem{example}[theorem]{Example}
  }
  {
  }

\begin{document}

\nocite{Sturm06a}

\title[Locality of the heat kernel]
{Locality of the heat kernel on metric measure spaces}

\author{Olaf Post}
\address{Fachbereich 4 -- Mathematik,
  Universit\"at Trier,
  54286 Trier, Germany}
\email{olaf.post@uni-trier.de}

\author{Ralf R\"uckriemen}
\address{Fachbereich 4 -- Mathematik,
  Universit\"at Trier,
  54286 Trier, Germany}
\email{rueckriemen@uni-trier.de}

\ifthenelse{\isundefined \draft}
           {\date{\today}}  
           {\date{\today, \thistime,  \emph{File:} \texttt{\jobname.tex}}} 

\begin{abstract}
We will discuss what it means for a general heat kernel on a metric measure space to be local. We show that the Wiener measure associated to Brownian motion is local. Next we show that locality of the Wiener measure plus a suitable decay bound of the heat kernel implies locality of the heat kernel. 
We define a class of metric spaces we call manifold-like that satisfy the prerequisites for these theorems. This class
  includes Riemannian manifolds, metric graphs, products and some
  quotients of these as well as a number of more singular
  spaces. There exists a natural Dirichlet form based on the Laplacian
  on manifold-like spaces and we show that the associated Wiener measure and heat kernel are both local. These results unify and generalise facts known for manifolds
  and metric graphs. They provide a useful tool for computing heat
  kernel asymptotics for a large class of metric spaces. As an
  application we compute the heat kernel asymptotics for two identical
  particles living on a metric graph.
\end{abstract}

\maketitle 


\thanks{We would like to thank Batu G\"uneysu for a fruitful
  discussion and helpful hints on a preliminary version of our
  manuscript.}

\tableofcontents

%
\section{Introduction and motivation}
\label{sec:intro}
%
\nocite{AGS13} 

This article is about the well-known statement: `The heat kernel is
local'. We will discuss what this means exactly and when it
holds. This statement first appeared around 1950, for example in the
works of~\cite{MinakshisundaramPleijel49}. They studied the short time
expansion of the heat kernel of the Laplace-Beltrami operator on
Riemannian manifolds. In order to go from boundaryless manifolds to
manifolds with boundary one uses the following trick. A small
neighbourhood of a point on the boundary looks like a half space where
the heat kernel is known explicitly, a small neighbourhood of a point
away from the boundary looks like a neighbourhood of a point in a
boundaryless manifold, where the expansion of the heat kernel is also
known. One can get an expansion of the heat kernel of a manifold with
boundary by combining these two situations. The fact that this
strategy works is the essence of what is meant by saying the heat
kernel is local. For that reason, the same fact is also
referred to as `the heat kernel does not feel the boundary', see \cite{Kac51}.

A first naive understanding of `the heat kernel is local' would be the
following. Given a space $M$ with some operator $D$ on it and some
nice set $U \subset M$, one could think that locality means the heat
kernel on $U$ is determined by knowing $U$ and $D|_U$. Unfortunately
this is not true. Consider a simple example, the unit interval with
the standard Laplace operator, once with Neumann boundary conditions
at both ends $\laplacianN$ and once with Dirichlet boundary conditions
$\laplacianD$. Let $U = (a, 1-a)$ for some $a>0$. One can compute both
heat kernels $p_{\laplacianN}$ and $p_{\laplacianD}$ explicitly and
$p_{\laplacianN} \restr U \neq p_{\laplacianD} \restr U$. For large
times one should not expect this to be true anyway, the Neumann Laplace
operator preserves energy, whereas in the Dirichlet case all heat
eventually leaks out of the system.

What is true however, is that the difference between the two heat
kernels is quite small for small times $t$, namely it can be bounded
by $\e^{-\eps/t}$ for some explicit constant $\eps>0$ depending
on $a$. In particular, such a bound implies that the asymptotic
expansion for $t \to 0$ of these two heat kernels is equal on
$U$. Hence, getting a bound of this form in a quite general setting is
our primary goal.

Locality in this sense holds for Riemannian manifolds with boundary
and the Laplace-Beltrami operator. It still holds if we add lower
order terms to the operator. It is also known for quantum graphs with
the Laplacian and some suitable boundary conditions at the
vertices. Some works on orbifolds also make indirect use of it. 

A slightly weaker notion of locality is proved in \cite{Hsu95} for Riemannian manifolds and in \cite{Gueneysu17} for general metric spaces. Their definition of the principle of not feeling the boundary only makes a statement about the behaviour in the $t \ra 0^+$ limit. Hence our locality definition implies theirs but not vice versa.

While in the original works from the 50s the locality statement is
first rigorously proven before it is used, in some of the more modern
works it seems to be the other way round. Some asymptotic expansion of
the heat kernel is derived, the derivation clearly makes use of the
some version of the locality of the heat kernel. Sometimes this
locality is just referred to as a well known fact, sometimes not even
that but little consideration is given to the question whether it
actually holds in the given setting and how this might be proven. Here
we are going to show that locality does indeed hold in a very general
setting.

Our strategy is based on the Wiener measure. One can represent the heat
kernel at time $t$ between two points $x$ and $y$ as the Wiener
measure of the set of paths from $x$ to $y$ in time $t$. In the unit
interval example the Neumann and the Dirichlet Laplacian give rise to
Wiener measures $W^\Neu$ and $W^\Dir$. These Wiener measures satisfy
locality in the literal sense of the word, that is $W^\Neu \restr U =
W^\Dir \restr U$, where the restriction means the restriction to paths
that stay inside of $U$ during the entire time interval $[0,t]$. Our
first \Mthm{locality_wiener_measure} states that locality of the
Wiener measure holds for fairly general metric measure spaces.

Our next \Mthm{locality_heat_kernel} states that if the heat kernels
satisfy a suitable decay bound (see \Def{heat_kernel_decay}), then locality of the Wiener measure
implies locality of the heat kernel. Decay bounds for general heat kernels on metric spaces are an area of active research, see for example \cite{Sturm95, GrigoryanTelcs12, GrigoryanKajino17}.

To show that this heat kernel decay and hence locality is satisfied on
a wide class of examples, we define a class of metric spaces we call
manifold-like. They are an extension of spaces that satisfy the
measure contraction property introduced by Sturm in~\cite{Sturm98}
and~\cite{Sturm06b}. This class includes Riemannian manifolds with our
without boundary, quantum graphs, products and certain quotients of
these and a number of other spaces. On these spaces there exists a
natural Dirichlet form, and the associated operator corresponds to the
Laplace operator in the case of Riemannian manifolds. The associated
heat kernel satisfies a bound of the form $p_t(x,y) \le
\e^{-d(x,y)^2/(c\cdot t)}$, that is the heat kernel decays
exponentially fast with distance between the points. Our third
\Mthm{manifold_like_spaces} shows that locality of the heat kernel
holds for manifold-like spaces.

This paper is organised as follows. \Sec{gen.heat.kernels} defines
heat kernels and the Wiener measure in a metric space setting and
collects some facts on how the operator, the Dirichlet form, the
semigroup, the heat kernel and the Wiener measure all relate to each
other and imply each others existence. \Sec{locality_wiener_measure} contains our
\Mthms{locality_wiener_measure}{locality_heat_kernel},
locality of the Wiener measure and locality of the heat kernel on metric measure spaces.
In \Sec{the.space} we define
our class of manifold-like metric spaces and in \Sec{dir.forms} we
define Laplace-type operators through Dirichlet forms on them. We then show that these operators satisfy the conditions for locality of the Wiener measure and the heat kernel in \Mthm{manifold_like_spaces}.
Finally, \Sec{example} contains an example computation of the heat
kernel asymptotics of a two particle system on a metric graph through
decomposing the state space into simple pieces.

%
\section{General heat kernels}
\label{sec:gen.heat.kernels}
%



A heat kernel can arise in a variety of ways, one can define a heat
kernel directly, it can arise as the fundamental solution of a
suitable operator, which itself can be defined through a Dirichlet
form, it can be seen as the transition function of a Markov process or
as the integral kernel of a semigroup. In full generality not all
these interpretations might exist but we are interested in situations
where they are. We will first collect some well-known facts that hold
in great generality.

To introduce a heat kernel and a Markov process we need an underlying
space $M$ with Hilbert space $\Lsqr M$. Throughout the paper we let
\emph{$(M,d)$ be a \textbf{compact} metric space} ($M$ is then automatically
complete and separable).  The metric is a map $\map d {M \times
  M}{[0,\infty]}$, in particular we do not assume that $M$ is
connected. Let $\mu$ be a Radon measure on $\mathcal{B}(M)$, the Borel
sets on $M$. We assume $\mu$ has full support on $M$. The triple
$(M,d,\mu)$ is called a \emph{metric measure space}.

A semigroup $P_t$ is called \emph{Markovian} if for any $f$ satisfying
$0 \le f \le 1$ $\mu$-almost everywhere we also have $0 \le P_tf \le
1$ $\mu$-almost everywhere. It is \emph{contracting} if $\norm {P_t}
\le 1$.

A closed, densely defined linear operator $D$ is called a
\emph{Dirichlet operator} if for all $f \in \dom D$ we have
$\iprod{Df}{\max \{f-1, 0\}} \ge 0$.

As usual, we use the same symbol for a \emph{symmetric form} $\qf E$
and the associated \emph{quadratic form} given by $\qf E(f):= \qf
E(f,f)$.  We assume that $\qf E$ is non-negative, i.e., $\qf E(f) \ge
0$ for all $f \in \dom \qf E$.  Such a form is \emph{closed} if $\dom
\qf E$ equipped with the norm given by $\normsqr[\qf E] f :=
\normsqr[\Lsqr M] f + \qf E(f)$ is complete, i.e., itself a Hilbert
space.  A closed symmetric form $ \qf E$ is a \emph{Dirichlet form} if
the unit contraction operates on it, i.e., if for $f \in \dom \qf E$,
then $f^\# \in \dom \qf E$ and $\qf E(f) \ge \qf E(f^\#)$, where
$f^\#:= \min( \max (f, 1), 0)$ denotes the \emph{unit contraction
  operator}.

A Dirichlet form is called \emph{regular} if $\Cont M \cap \dom \qf E$
is dense in $\Cont M$ (with respect to the supremum norm) and also
dense in $\dom \qf E$ with respect to $\norm[\qf E]\cdot$.  A
Dirichlet form is called \emph{local} if $\qf E(f,g)=0$ whenever $f,g
\in \dom \qf E$ have disjoint support.  If $\qf E(f,g)=0$ whenever
$f$ is constant on a neighbourhood of the support $\supp g$ of $g$,
then $\qf E$ is called \emph{strongly local}.

The following is well known, see for example~\cite{Fukushima80},
\cite{Kato80},~\cite{Davies80},~\cite{FOT11},~\cite{MaRoeckner92}
or~\cite{BouleauHirsch91}.
 
\begin{theorem}
  \label{thm:semigroup_form_operator}
  Let $H$ be a Hilbert space. Then the existence of any of the
  following implies the existence of the others.
  \begin{itemize}
  \item a linear non-negative self-adjoint and densely defined
    operator $D$\item a closed symmetric form $\qf
    E$ 
  \item a strongly continuous contraction semigroup $\{P_t\}_{t \ge
      0}$ of self-adjoint operators
  \end{itemize}
  If $H=\Lsqr M$ we can additionally say the following.  The semigroup
  satisfies the Markov property if and only if the operator is a
  Dirichlet operator if and only if the symmetric form is a Dirichlet
  form.
  
\end{theorem}
\begin{proof}
  The equivalence of operator and form is given by $\iprod {Df} g =
  \qf E(f,g) $ and $\dom \qf E = \dom \sqrt D$. The operator $D$ is
  the generator of the semigroup $P_t:= \e^{tD}$.
\end{proof}

\begin{definition}
  A family $\{p_t\}_{t>0}$ of functions $\map {p_t}{M \times M} \R$ is
  called a \emph{heat kernel} if it satisfies the following conditions
  for all $t>0$.
  \begin{enumerate}
  \item Measurability: $p_t(\cdot,\cdot)$ is $\mu \times \mu$ measurable
  \item Markov property: $p_t(x,y) \ge 0$ for $\mu$-almost all $x,y$
    and $\int_M p_t(x,y) \dd \mu(x) = 1$ for $\mu$-almost all $y$
  \item Symmetry: $p_t(x,y) = p_t(y,x)$ for $\mu$-almost all $x,y$
  \item Semigroup property: for any $s,t>0$ and $\mu$-almost all $x,y$
    we have
    \begin{align*}
      p_{s+t}(x,y) = \int_M p_s(x,z)p_t(z,y) \dd\mu(z)
    \end{align*}
  \item Approximation of identity: for any $f \in \Lsqr M$ we have
    \begin{align*}
      \int_M p_t(x,y)f(y) \dd\mu(y) \to_{t \to 0^+} f(x)
    \end{align*}
  \end{enumerate}
\end{definition}

\begin{remark}
  \label{rem:compactification}
  One can consider more general heat kernels that satisfy only the
  sub-Markov property $\int_{M} p_t(x,y) \dd\mu(x) \le 1$. In this
  case one usually expands the space $M$ by a cemetery point $\Delta$
  and then extends the heat kernel so that it satisfies the Markov
  property on this larger space. In a similar way one can generalise
  the definition of a Markov process below. This construction is well
  known and explained for example in~\cite{FOT11}. However it makes
  various formulations a lot more clunky and gets somewhat
  technical. In the interest of clear exposition we will restrict to
  the strict Markov property here but remark that the generalization
  to sub-Markov is straight forward.
\end{remark}

\begin{remark}
  Heat kernels are only defined up to behaviour on a set of
  $\mu$-measure zero. We will thus identify heat kernels that agree
  $\mu$-almost everywhere.

  If the heat kernel is continuous, this distinguishes this particular
  heat kernel, so that all the $\mu$-almost everywhere statements can
  be replaced by everywhere statements.
\end{remark}

\begin{definition}
  \label{def:hunt-process}
  A \emph{Markov process} on the set of continuous paths $\mathcal{P}(M)$
  consists of a family of probability measures $\{\Proba_x\}_{x \in
    M}$ on $\mathcal{P}(M)$ such that $\Proba_x\left( \omega(0)=x
  \right) =1$ and a stochastic process $X_t(\omega):=\omega(t)$ on
  $\mathcal{P}(M)$ with values in $M$. It additionally satisfies the
  Markov property. See the appendix for more details.
  
  If it satisfies the \emph{strong Markov property},
  that is
  \begin{align*}
    \mathbb{P}_x(X_{\zeta+s} \in U | \mathcal{F}_{\zeta}) =
    \mathbb{P}_{X_{\zeta}}(X_s \in U)
  \end{align*}
  holds for any time stopping function $\zeta$, then it is called a
  \emph{continuous Hunt process} or a \emph{diffusion}. 
\end{definition}
By convention one refers to $X_t$ as the Hunt process, the probability
measure is implicit.

\begin{theorem}[\cite{Grigoryan03}]
  A heat kernel $p_t$ defines a semigroup via
  \begin{align*}
    P_t f(x) := \int_M p_t(x,y)f(y) \dd\mu(y)
  \end{align*}
  This semigroup is strongly continuous, self-adjoint, contracting and
  Markov.
\end{theorem}
The Markov property is not mentioned in~\cite{Grigoryan03}, but it
follows trivially from the definition.

\begin{theorem}[\cite{BlumenthalGetoor68}] 
  \label{thm:markov_transition}
  A continuous Hunt process $X_t$ on $M$ defines a heat kernel via
  \begin{align*}
    \int_U p_t(x,y)d\mu(y) = \Proba_x(X_t \in U)
  \end{align*}
  The heat kernel is the density function of the transition function
  $p_t(x,U):=\Proba_x(X_t \in U)$. See the appendix for details.
\end{theorem}

Two stochastic processes are called \emph{equivalent} if their
transitions functions agree outside of a properly exceptional
set. Note that all properly exceptional sets have $\mu$-measure zero.

\begin{theorem}[{\cite[Thm 7.2.1 and Thm 4.2.8]{FOT11}}]
  \label{Dirichlet_Hunt}
  Given a regular local Dirichlet form $\qf E$ on $\Lsqr M$, there
  exists a continuous Hunt process $X_t$ on $M$ whose Dirichlet form
  is the given one $\qf E$.
  This Hunt process is unique up to equivalence.
\end{theorem}

\begin{remark}
  This theorem is by far the most difficult part in the
  equivalence. The full proof goes over several dozen pages. The basic
  construction is as follows. Given a family of probability measures
  $p_t(x,U)$, for $x$ fixed and $t \in [0,T]$, the Kolmogorov
  extension theorem guarantees the existence of a stochastic process
  $X_t$ and a probability measure $\Proba_x$ such that
  \begin{align*}
    p_t(x,U) = \Proba_x(X_t \in U)
  \end{align*}
  but proving that this process has the claimed regularity is very
  hard.
\end{remark}

%
\section{Locality of the Wiener measure and the heat kernel}
\label{sec:locality_wiener_measure}
%

In this section we will show that the Wiener measure is local and then that the heat
kernel is local provided it satisfies a suitable decay bound. We will first introduce the
notion of martingales and then quote a uniqueness and existence
theorem for the Wiener measure.  Next, if two spaces are identical on
some subset, we can define a new measure on the set of paths of one of the spaces by using
one of the measures inside the subset and the other one outside. This
is called splicing and will be explained in further detail. One can
then show that this spliced Wiener measure is also compatible with the
operator, by uniqueness this implies that the spliced measure is
identical to the original measure. In other words, on the subset where
the spaces and operators agree, so do the Wiener measures. Combined
with a decay bound (see \Def{heat_kernel_decay}) this implies that
the heat kernel is local as well.

\subsection{Local isometries}
\label{ssec:local.isom}

Let $(M,d,\mu)$ and $(\oS M,\oS d, \oS \mu)$ be two metric measure
spaces with energy forms $\qf E$ and $\oS{\qf E}$ and associated
operators $D$ and $\oS D$, respectively. 

Assume that $U \subset M$ and $\oS U \subset \oS M$ are open and that there exists a 
local isometry $\map \psi U {\oS U=\psi(U)}$.
For a function  $\map f M \R$ with
$\supp f \subset U$ we denote by $\psi_*f$ the function $f\restr U
\circ \psi^{-1}$ extended by $0$ onto $\oS M$.  Note that
$\map{\psi_*}{\Lsqr U}{\Lsqr{\oS U}}$ is unitary.
\begin{definition}
  \label{def:forms.ops.agree}
  \indent
  \begin{enumerate}
  \item
    \label{agree.a}
    We say that $\qf E$ and $\oS{\qf E}$ \emph{agree on $U$ and $\oS
      U$} if there is a measure preserving isometry $\map \psi U {\oS
      U=\psi(U) \subset \oS M}$ such that for any $f \in \dom \qf E$
    with $\supp f \subset U$ we have $\psi_*f \in \dom \oS{\qf E}$ and
    $\qf E(f) = \oS{\qf E}(\psi_* f)$.

  \item 
    \label{agree.b}
    Similarly, we say that $D$ and $\oS D$ \emph{agree on $U$ and $\oS
      U$} if there is a measure preserving isometry $\map \psi U {\oS
      U=\psi(U) \subset \oS M}$ such that for any $f \in \dom D$ with
    $\supp f \subset U$ we have $\psi_*f \in \dom \oS D$ and
    $\psi_*(Df)=\oS D(\psi_*f)$.
  \end{enumerate}
\end{definition}

\begin{lemma}
  \label{lem:forms.ops.agree}
  Assume that $\qf E$ and $\oS {\qf E}$ are local Dirichlet forms,
  then $\qf E$ and $\oS {\qf E}$ agree on $U$ and $\oS U$ if and only
  if $D$ and $\oS D$ agree on $U$ and $\oS U$.
\end{lemma}
\begin{proof}
  Note first that $\qf E$ is local (i.e., $\qf E(f,g)=0$ for all $f,g
  \in \dom \qf E$ with $\supp f \cap \supp g = \emptyset$) if and only
  if $D$ is local (i.e., $\supp Df \subset \supp f$ for all $f \in
  \dom D$).

  Let $f \in \dom D$ with $\supp f \subset U$.  Then there is an
  open set $V$ such that $\supp f \subset V \subset \clo V \subset U$.
  If $g \in \dom \qf E$ with $\supp g \subset U$, then
  \begin{equation*}
    \iprod[\Lsqr {\oS M}]{\psi_*(Df)} {\psi_*g}
    = \iprod[\Lsqr M]{Df} g
    = \qf E(f,g)
    = \oS{\qf E}(\psi_*f,\psi_*g)
  \end{equation*}
  as $\psi_*$ is an isometry for functions with support in $U$ and
  $\oS U$ (we used the locality of $D$ here) and as $\qf E$, $\oS{\qf
    E}$ agree on $U$ and $\oS U$. 
  
  For $\oS g \in \dom \oS{\qf E}$ with $\supp \oS g \cap V =
  \emptyset$, we have $\iprod[\Lsqr {\oS M}]{\psi_*(Df)} {\oS g}=0$
  (again by locality of $D$) and $\oS{\qf E}(\psi_*f,\oS g)=0$.  Since
  all $\psi_*g$ with $g \in \dom \qf E$ and $\supp g \subset U$ and
  $\oS g \in \dom \oS{\qf E} \cap V = \emptyset$ span $\dom \oS{\qf
    E}$, we have shown that $\psi_* f \in \dom \oS D$ and $\oS
  D(\psi_*f)=\psi_*(Df)$.

  The opposite implication can be seen similarly.
\end{proof}

\subsection{Martingales}

Recall that $\mathcal P(M)$ denotes the set of continuous paths on a
metric measure space $(M,d,\mu)$.
\begin{definition}
  A stochastic process $\map Y {[0, \infty) \times \mathcal{P}(M)} \R$
  is called a \emph{martingale} with respect to the family of
  probability measures $\Proba=\{\Proba_x\}_{x\in M}$ and the
  increasing sequence of $\sigma$-algebras $\mathcal{F}_t$ if the
  following conditions are fulfilled:
  \begin{enumerate}
  \item Measurability: $Y(t, \cdot)$ is $\mathcal{F}_t$ measurable.
  \item Right continuity: for every $\omega \in \mathcal{P}(M)$ the
    map $t \mapsto Y(t, \omega)$ is right continuous.
  \item Conditional constancy: for $0 \le s < t$ we have
    \begin{align*}
      Y(s, \cdot) = \mathbb E^{\Proba_x}[Y(t, \cdot) | \mathcal{F}_{s}]
    \end{align*}
    holds $\Proba_x$-almost surely.
  \end{enumerate}
\end{definition}
In most of our applications the probability measures and the
$\sigma$-algebras will come from a Markov process $X_t$,
$\mathcal{F}_t=\sigma(X_s| s\le t)$ and we will just write $Y$ is a
martingale with respect to $X_t$.

\begin{definition}
  Let $X_t$ be a Markov process on $\mathcal{P}(M)$ and let $D$ be a
  non-negative self-adjoint operator on $M$. For each $f \in \dom D$,
  we define a stochastic process $\map{M_f}{[0,\infty) \times
    \mathcal{P}(M)} \R$ by setting
 \begin{align*}
   M_f (t, X) := f(X_t) - \int_0^t (D f)(X_s) \dd s.
 \end{align*}
 We say that the Markov process $X_t$ \emph{solves the martingale
   problem} for $(M,D)$ if $M_f$ is a martingale with respect to $X_t$
 for each $f \in \dom D$.
 \end{definition}

 \begin{theorem}[\cite{EthierKurtz86}]
   \label{thm:uniqueness_martingales}
   Let $(M,d,\mu)$ be a compact metric measure space and $D$ a
   non-negative self-adjoint operator on it.
   Then the continuous Hunt process $X_t$ associated to $D$ is the
   unique solution of the martingale problem for $(M,D)$.
\end{theorem}

\subsection{Splicing measures}

We will follow the construction of~\cite{Stroock05} for splicing
measures on $\R^n$ and extend it to the more general setting of metric
measure spaces.

\begin{definition}
  \label{def:stopping-time}
  A measurable function $\map \zeta {\mathcal{P}(M)}{[0,\infty]}$ such
  that for all $t \ge 0$, $\{\zeta \le t\} \in \mathcal{F}_t^0$, is
  called a \emph{stopping time function}.

  Let $\omega \in \mathcal{P}(M)$ and $U \subset M$ be open, let
  \begin{align*}
    \zeta_U(\omega) := \inf \{t \ge 0 | \omega(t) \notin U\}
  \end{align*}
  be the first exit time from $U$ of the path $\omega$. The function
  $\zeta_U$ is an example of a stopping time function. All stopping
  time functions we are going to use are of this form.
\end{definition}

\begin{definition}[Splicing measures on the same space]
  \label{def:spliced_measures}
  Let $\oS \Proba=\{\oS \Proba_x\}_{x \in M}$ be a family of Borel
  probability measure on $\mathcal{P}(M)$ with $\oS
  \Proba_x(\omega(0)=x)=1$. Let $U \subset M$ be open and let $\chi_U$
  denote the characteristic function of $U$. Define a family of Borel
  probability measures $\delta_{\hat\omega} \otimes_t \oS \Proba$ on
  $\mathcal{P}(M)$ indexed by $t \in [0,\infty)$ and $\hat\omega \in
  \mathcal{P}(M)$ by setting
  \begin{align*}
    \left( \delta_{\hat\omega} \otimes_t \oS \Proba\right) 
          \left(\omega(s) \in U \right) 
    := \begin{cases}
      \chi_U(\hat\omega(s)) & s < t \\
      \oS \Proba_{\hat\omega(t)}(\omega(s-t) \in U)  & s \ge t
    \end{cases}
  \end{align*}
  
  Next, given another family of Borel probability measures $\Proba$ on
  $\mathcal{P}(M)$ and a stopping time function $\map \zeta
  {\mathcal{P}(M)} {[0,\infty]}$ define a new family of spliced
  measures $\Proba \otimes_{\zeta} \oS \Proba$ by setting
  \begin{align*}
    & (\Proba \otimes_{\zeta} \oS \Proba)_x(\omega(t) \in U) \\
    :=& \int_{\{\hat\omega \in \mathcal{P}(M) |\zeta(\hat\omega) < \infty\}}
           \delta_{\hat\omega} \otimes_{\zeta(\hat\omega)} 
        \oS \Proba(\omega(t) \in U) \dd\Proba_x(\hat\omega) 
    + \Proba_x(\omega(t) \in U | \{ \zeta(\omega) = \infty \}).
  \end{align*}
\end{definition}
This can be interpreted as follows. Each path $\omega$ is measured
with $\Proba$ until time $\zeta(\omega)$. After time $\zeta(\omega)$
it is measured by $\oS \Proba$ shifted back in time by
$\zeta(\omega)$.

\begin{remark}
  This spliced measure $\Proba \otimes_{\zeta} \oS \Proba$ is also
  completely determined by stating that $\Proba \otimes_{\zeta} \oS
  \Proba|_{\mathcal{F}_{\zeta}}= \Proba|_{\mathcal{F}_{\zeta}}$ and
  that the conditional distribution of shifted paths $\omega(\tau)
  \mapsto \omega(\tau+\zeta(\omega))$ under $\Proba \otimes_{\zeta}
  \oS \Proba$ with $\mathcal{F}_{\zeta}$ given, is just $\oS
  \Proba_{\omega(\zeta)}$.
\end{remark}

\begin{remark}
  Note that if the stopping time function is of the form $\zeta_U$
  defined above, paths that leave $U$ but reenter it at a later point
  would be measured with $\oS \Proba$ upon reentering. Hence the
  spliced measure is not just using one measure inside the set $U$ and
  the other one outside of it.
\end{remark}

\begin{definition}[Splicing measures on different spaces]
  Let $(M,d,\mu)$ and $(\oS{M},\oS{d},\oS{\mu})$ be two metric measure spaces. Assume there
  exists an open set $U \subset M$ and a measure preserving isometry
  $\map \psi U {\psi(U)} \subset \oS{M}$. Let $\Proba$ and $\oS
  \Proba$ be two families of Borel probability measures on
  $\mathcal{P}(M)$ and
  $\mathcal{P}(\oS{M})$. 
  
  For $A \subset U$ and $x \in U$ we let
  \begin{align*}
    \Proba^U_x(\omega(t) \in A) := \oS \Proba_{\psi(x)}(\psi(\omega(t))
    \in \psi(A))
  \end{align*}
  This defines a family of Borel measures on $\mathcal{P}(U)$.
\end{definition}
We can now define the spliced measure $\Proba^U \otimes_{\zeta_U}
\Proba$ which is a family of Borel probability measures on
$\mathcal{P}(M)$ as in \Def{spliced_measures}.

\subsection{Locality of the Wiener measure} 

For $a,b \in \R$ let $a \wedge b:= \min(a,b)$.
\begin{theorem}[{Doob's time stopping theorem~\cite{Stroock11}}]
  \label{thm:doob_time_stopping}
  If $Y(t,\omega)$ is a martingale with respect to a Markov process
  $X_t$, then for any time stopping function $\zeta$, $Y(t \wedge
  \zeta(\omega),\omega)$ is also a martingale with respect to the
  Markov process $X_t$.
\end{theorem}

\begin{lemma}[{\cite{StroockVaradhan79}}] 
  \label{lem:glueing_martingales}
  Let $\zeta$ be a stopping time and $X_t$ a Markov process. Recall
  that $\mathcal{F}_t^0=\sigma(X_s| s\le t)$ and let
  $\{Q_{\hat\omega}\}_{\hat\omega \in \mathcal{P}(M)}$ be the
  conditional probability distribution of $\Proba$ with respect to
  $\mathcal{F}_t^0$ (see \Def{cond-proba}).  Let $\map M {[0,\infty)
    \times \mathcal{P}(M)} \R$ be a stochastic process. Assume $M$ is
  $\Proba$-integrable, $M(t, \cdot)$ is $\mathcal{F}_t^0$-measurable
  and $M(\cdot, \omega)$ is continuous.  Then the following two
  statements are equivalent:
  \begin{enumiii}
  \item $M(t,\omega)$ is a martingale with respect to $X_t$.
  \item $M(t \wedge \zeta(\omega), \omega)$ is a martingale with
    respect to $X_t$ and $M(t,\omega)-M(t \wedge
    \zeta(\omega),\omega)$ is a martingale with respect to the
    measures $Q_{\hat\omega}$ and the $\sigma$-algebra
    $\mathcal{F}_t^0$ for all $\hat\omega \in \mathcal{P}(M)$ outside
    of a $\Proba$-null-set.
  \end{enumiii}
\end{lemma}

Recall that two (local) operators $D$ and $\oS D$ agree on some
subsets if there is a measure preserving local isometry intertwining
$D$ and $\oS D$ (see \Defenum{forms.ops.agree}{agree.b}). 

We now formulate our first main theorem:
\begin{maintheorem}[Locality of the Wiener measure]
  \label{mthm:locality_wiener_measure}
  Let $(M,d,\mu)$ and $(\oS{M},\oS{d},\oS{\mu})$ be two metric measure spaces with non-negative
  self-adjoint operators $D$ and $\oS D$.
  Let $\Proba$ and $\oS \Proba$ be the associated Wiener measures.
  Assume that $D$ and $\oS D$ agree on some open subsets $U \subset M$
  and $\oS U \subset \oS M$, then
  \begin{align*}
    \Proba = \Proba^U \otimes_{\zeta_U} \Proba.
  \end{align*}
  In words, the spliced measure that uses the Wiener measure from $\oS
  M$ until the first exit from $U$ and the Wiener measure from $M$
  after the first exit time is identical to the original Wiener
  measure on $M$.
\end{maintheorem}

\begin{corollary}
  \label{cor:main1}
  Under the assumptions of the previous theorem we have
  \begin{align*}
    \Proba\restr{\mathcal{F}(U)} = \oS \Proba \restr{\mathcal{F}(\oS U)},
  \end{align*}
  i.e., when restricted to paths that stay inside $U$, the two Wiener
  measures are identical.
\end{corollary}
\begin{proof}[Proof of \Mthm{locality_wiener_measure}]
  This proof is a generalization of a proof in~\cite{Stroock11} where Stroock
  shows the above theorem for $\R^n$ instead of metric spaces.

  We are going to show that the Markov process
  $X_t(\omega):=\omega(t)$ with the family of measures $\Proba^U
  \otimes_{\zeta_U} \Proba$ solves the martingale problem for
  $D$. Then uniqueness of the solution (\Thm{uniqueness_martingales})
  shows the equality of the measures. Thus we need to check that for
  all $f \in \dom D$ the map
  \begin{align*}
    M_f (t, \omega) = f(\omega(t)) - \int_0^t (D f)(\omega(s)) ds
  \end{align*}
  is a martingale with respect to $\Proba^U \otimes_{\zeta_U}
  \Proba$.

  We let $\oS f = f \circ \psi^{-1}$ and define
  $\oS M_{\oS f}$ analogously to $M_f$, hence
  $\oS M_{\oS f}(t,\oS \omega)$ is a
  martingale with respect to $\oS{X}_t$. Through the isometry
  $\psi$ we get $M_f(t \wedge \zeta_U(\omega), \omega)=
  \oS M_{\oS f}(t \wedge
  \zeta_{\psi(U)}(\oS \omega), \oS \omega)$ which
  is a martingale with respect to $\oS{X}_t$ by Doob's time
  stopping \Thm{doob_time_stopping}.

  Note that up to time $\zeta_U(\omega)$ the measures
  $\oS \Proba$ and $\Proba^U \otimes_{\zeta_U}
  \Proba$ are identical, so this implies that $M_f(t \wedge
  \zeta_U(\omega), \omega)$ is a martingale with respect to
  $\Proba^U \otimes_{\zeta_U} \Proba$ as well.

  $M_f(t,\omega) - M_f(t \wedge \zeta_U(\omega), \omega)$ is just the
  function $M_f(t,\omega)$ starting at time $\zeta_U(\omega)$.
   Hence $M_f(t,\omega) - M_f(t \wedge
  \zeta_U(\omega), \omega)$ is a martingale for a shifted version of
  some $\overline{\Proba}$ for $t \ge 0$ if and only if
  $M_f(t,\omega)$ is a martingale for $\overline{\Proba}$ for $t
  \ge \zeta_U(\omega)$. Here being a martingale for $t \ge
  \zeta_U(\omega)$ means that the conditionally constant property only
  holds for these $t$ and not for $t \ge 0$ as in the original
  definition. For $t \ge \zeta_U(\omega)$, the measure
  $\delta_\omega \otimes_{\zeta_U(\omega)}\Proba$ is the time
  shifted version of $\Proba$. Thus $M_f(t,\omega) - M_f(t \wedge
  \zeta_U(\omega), \omega)$ is a martingale with respect to
  $\delta_\omega \otimes_{\zeta_U(\omega)}\Proba$ if and only if
  $M_f(t,\omega)$ is a martingale with respect to $\Proba$. The
  latter is true by assumption.

  Next we have that $\delta_\omega
  \otimes_{\zeta_U(\omega)}\Proba$ is the conditional probability
  distribution of $\Proba^U \otimes_{\zeta_U} \Proba$ with
  respect to $\mathcal{F}_{\zeta_U(\omega)}$, which by definition
  means
  \begin{align*}
    (\Proba^U \otimes_{\zeta_U} \Proba)(\omega(s) \in A \cap
    B) 
    = \int_A (\delta_\omega \otimes_{\zeta_U(\omega)}\Proba)(\omega(s) \in B)
    \dd\Proba(\omega)
  \end{align*}
  for $A \in \mathcal{F}_t$ and $B \in \mathcal{F}$. This is just the
  definition of the spliced measure \Def{spliced_measures}.

  Hence we can apply \Lem{glueing_martingales} and conclude that
  $M_f$ is a martingale with respect to the measures $\Proba^U
  \otimes_{\zeta_U} \Proba$.
\end{proof}

\subsection{Locality of the heat kernel}
\label{sec:locality.heat.kernel}

Here we are going to prove that exponential decay of the heat kernel together with
the locality of the Wiener measure imply locality of the heat
kernel.

\begin{definition}
  \label{def:heat_kernel_decay}
  Let $(M,d,\mu)$ be a metric measure space. Let $\qf E$ be a strongly
  local regular Dirichlet form on it. Then we say the heat kernel satisfies an \emph{exponential 
  decay bound} if there exist constants $C,c>0$ and $n \ge 0$ such that 
  \begin{align*}
    p_t(x,y) \le C t^{-n/2} \exp\left(-\frac{d^2(x,y)}{ct}\right)
  \end{align*} 
  holds for all $x,y \in M$ and all $t$ as long as
  $0<t<T$ for some $T$.
\end{definition}
In most applications $n$ is the dimension of $M$.

\begin{theorem}[Existence of conditional Wiener
  measure~\cite{BaerPfaeffle11}]
  \label{thm:bp11}
  Let $(M,d,\mu)$ be a metric measure space. Assume the heat kernel
  $p$ satisfies an exponential decay bound as in \Def{heat_kernel_decay}.
  Then there exists a unique point-to-point Wiener measure
  $\Proba_x^y$ on the set $\Contsymb_x^y([0,T],M)$, that is the set of
  continuous paths with start point $x$ and end point $y$ at time $T$.
  Moreover, it satisfies
  \begin{align*}
    \Proba_x^y(\omega(t) \in U) = \int_U p_t(z,y)p_{T-t}(x,z)
    \dd\mu(z)
  \end{align*}
  for $t \in (0,T)$ and $U \in \mathcal{B}(M)$. It is compatible with
  the Wiener measure $\Proba_x$ in the sense that
  \begin{align*}
    \int_{\Contsymb_{x}([0,t],M)} f(\omega) \dd\Proba_{x}(\omega) =
    \int_{M} \int_{\Contsymb_{x}^{y}([0,t],M)} f(\omega)
    \dd\Proba_{x}^{y}(\omega) \dd\mu(y)
  \end{align*}
  for any function $\map f {\Contsymb_{x}([0,t],M)} \R$ that is
  integrable with respect to $\Proba_{x}$.
\end{theorem}
 
\begin{remark}
  Note that the definition of $\Proba_x^y$ immediately implies
  \begin{align*}
    \Proba_x^y(\Contsymb_x^y([0,t],M)) = p_t(x,y).
  \end{align*}
  \Thm{bp11} cited above from~\cite{BaerPfaeffle11} requires a decay
  bound of the heat kernel that is much weaker than the one we use
  here.
\end{remark}

\begin{definition}
  For $U \subseteq M$ be open and $p_t$ a heat kernel on $M$, let
  \begin{align*}
    p_t^U(x,y):= \Proba_x^y(\Contsymb_x^y([0,t],U)).
  \end{align*}
  This means $p_t^U$ kills off all paths that leave the set $U$ and
  corresponds to the heat kernel on $U$ with Dirichlet boundary
  conditions. In particular $p_t^U(x,y) \le p_t(x,y)$.
\end{definition}

\begin{proposition}
  Let $(M,d,\mu)$ and $(\oS M, \oS d, \oS \mu)$ be two metric measure
  spaces with non-negative self-adjoint operators $D$ and $\oS D$ and let $p_t$ and $\oS p_t$ be the associated heat kernels.
  Assume $D$ and $\oS D$ agree on some
  open subsets $U$ and $\oS U$ via a measure preserving isometry $\map
  \psi U {\oS U}$. 
  Then 
  \begin{align*}
   p_t^U(x,y) = p_t^{\prime \psi(U)}(\psi(x),\psi(y)).
  \end{align*}
\end{proposition}
\begin{proof}
Using \Thm{bp11} this is exactly the statement of \Cor{main1}.
\end{proof}

The following lemma is based on an argument of~\cite{Hsu95}. The authors are
indebted to Batu G\"uneysu for providing us with this reference and
useful comments on the proof.

\begin{lemma}
\cite{Hsu95}
  \label{lem:heat_kernel_decomposition}
  Let $\zeta := \zeta_U$ be the time stopping function for the first
  exit time from the open set $U \subseteq M$. Let $x,y \in U$.  Then
  the following decomposition of the heat kernel
  \begin{align*}
    p_t(x,y) 
    = p_t^U(x,y) + \int_{\{ \omega \in \mathcal{P}_x(M) | \zeta(\omega) \le t\}} 
       p_{t- \zeta(\omega)}(\omega(\zeta(\omega)), y) \dd\mathbb{P}_x(\omega)
  \end{align*}
  holds for $\mu$-almost all $x,y \in U$.
\end{lemma}
This can be interpreted as follows. The set of all paths from $x$ to
$y$ in time $t$ is decomposed into the set of paths that stay inside
$U$ and those that do not. The paths that leave the set $U$ can be
represented as an integral using the time $\zeta(\omega)$ and place
$\omega(\zeta(\omega))$ where they leave the set $U$ for the first
time.

If the heat kernel is continuous as a function of $x$ and $y$ one can replace the $\mu$-almost all $x,y \in U$ by all $x,y \in U$.
\begin{proof}
  Let $f \in \Cont M$ be such that $\supp f \subset U$. Then we have
  \begin{align*}
    & \int_U f(y) p_t(x,y) \dd \mu(y)\\
    =& \int_{\mathcal{P}_x(M)} f(\omega(t)) \dd\mathbb{P}_x(\omega) \\
    =& \int_{\zeta > t} f(\omega(t)) \dd\mathbb{P}_x(\omega) 
    + \int_{\zeta \le t} f(\omega(t)) \dd\mathbb{P}_x(\omega) \\
    =& \int_U f(y) p_t^U(x,y) \dd\mu(y) 
    + \mathbb{E}^x[ f(\omega(t))\chi_{\zeta \le t}(\omega) ]
  \end{align*}
  by the definition of $p^U$. Here $\chi_{\zeta \le t}$ denotes the
  characteristic function of the set $\{\zeta \le t\}$ in
  $\mathcal{P}_x(M)$.

  Using the substitution $\tau:=t-\zeta$ we can write
  \begin{align*}
    & \mathbb{E}^x[ f(\omega(t))\chi_{\zeta \le t}(\omega) ] \\
    =& \mathbb{E}^x[ f(\omega(\tau+\zeta))\chi_{\tau \ge 0}(\omega) ] \\
    =& \mathbb{E}^x\left[\chi_{\tau\ge 0}(\omega) 
        \mathbb{E}^{\mathcal{F}_{\zeta}}\left[f(\omega(\tau+\zeta)\right]\right]  \\
    =& \mathbb{E}^x\left[\chi_{\tau\ge 0}(\omega) \mathbb{E}^{\zeta}\left[f(\omega(\tau)\right]\right] \\
    =& \int_{\mathcal{P}_x(M)}\chi_{\tau\ge 0}(\omega) 
         \mathbb{E}^{\zeta}[f(\omega(\tau)] d\mathbb{P}_x(\omega)
  \end{align*}
  where we first used the fact that the condition expectation
  $\mathbb{E}^{\mathcal{F}_{\zeta}}$ is just the identity projection
  in this case and then applied the strong Markov property (see
  \Def{strong-markov}).

  We have 
  \begin{align*}
    & \int_{\mathcal{P}_x(M)} \mathbb{E}^{\zeta}[f(\omega(\tau))]
    \chi_{\tau \ge 0}(\omega) \dd \mathbb{P}^x(\omega) \\
    =& \int_{\mathcal{P}_x(M)} \mathbb{E}^{\zeta}[f(\omega(t-\zeta))]
    \chi_{\zeta \le t}(\omega) \dd\mathbb{P}^x(\omega) \\
    =& \int_{\mathcal{P}_x(M)} \int_{\mathcal{P}_x(M)} 
    f(\tilde{\omega}(t-\zeta)) \dd\mathbb{P}_{\omega(\zeta)}(\tilde{\omega}) 
    \chi_{\zeta \le t}(\omega) \dd\mathbb{P}^x(\omega) \\
    =& \int_{\mathcal{P}_x(M)} \int_M p_{t-\zeta}(\omega(\zeta),y)f(y)\dd\mu(y) 
    \chi_{\zeta \le t}(\omega) \dd\mathbb{P}^x(\omega) \\
    =&\int_U \int_{\{ \omega \in \mathcal{P}_x(M) | \zeta(\omega) \le t\}} 
    p_{t- \zeta}(\omega(\zeta), y) \dd\mathbb{P}_x(\omega)f(y) \dd\mu(y) 
  \end{align*}
  where we used Fubini's theorem in the last step.

    This holds for all $f \in \Cont M$ with $\supp f \subset U$, hence
  we proved the lemma for $\mu$-almost all $y$.
\end{proof}

\begin{lemma}
  \label{lem:nonlocal_wiener_measure_bound}
  Let $(M,d,\mu)$ be a metric measure space. Let $p_t$ be a heat kernel
  on it. Assume the heat kernel satisfies a decay bound as in
  \Def{heat_kernel_decay}. Let $U \subset M$ be open.
  Then for $\mu$-almost all $x,y \in U$ we have
  \begin{align*}
    \Proba_x^{y}\left(\Contsymb^{y}_x\left([0,t],M\right) 
      \setminus \Contsymb^{y}_x\left([0,t],U\right)\right)  
    < C t^{-\frac{n}{2}} \e^{-\rho^2/(ct)}
  \end{align*}
   where 
  \begin{align*}
    \rho:= \inf d(\{x,y\}, \bd U)
  \end{align*}
  is the infimum of the distance of $x$ and $y$ from the boundary and
  $C,c> 0$ are some constants that are independent of $x,y$ for all $x,y$ such that $\rho$ is bounded away from zero.
\end{lemma}
This is a bound on the set of paths from $x$ to $y$ in time $t$ that
leave the set $U$. If the set $U$ is geodesically convex, these paths
are longer than the distance realizing paths.
\begin{proof}
  We have
  \begin{align*}
    & \Proba_x^{y}\left(\Contsymb^{y}_x\left([0,t],M\right) 
      \setminus \Contsymb^{y}_x\left([0,t],U\right)\right) \\
    =& p_t(x,y) - p_t^U(x,y) \\
    =& \int_{\{ \omega \in \mathcal{P}_x(M) | \zeta(\omega) \le t\}} 
    p_{t- \zeta}(\omega(\zeta(\omega)), y) \dd\mathbb{P}_x(\omega)\\
    \le& C\int_{\{ \omega \in \mathcal{P}_x(M) | \zeta(\omega) \le t\}} 
      (t- \zeta(\omega))^{-\frac{n}{2}} e^{-d(\omega(\zeta(\omega)), y)^2/(c(t-\zeta(\omega))} \dd\mathbb{P}_x(\omega)\\
    \le& C\int_{\{ \omega \in \mathcal{P}_x(M) | \zeta(\omega) \le t\}} 
      (t- \zeta(\omega))^{-\frac{n}{2}} e^{-d(\partial U, y)^2/(c(t-\zeta(\omega))} \dd\mathbb{P}_x(\omega)
  \end{align*}
  where we used the heat kernel decomposition from
  \Lem{heat_kernel_decomposition} and then the heat kernel decay bound. 
  
  Let $f(t) := t^{-n/2}e^{-\alpha/t}$ with $\alpha>0$ and let $T>0$ be fixed. Then for any $0 < s < t < T$ we have 
  \begin{align*}
   f(s) < \frac{f_{\max}}{f(T)} f(t)
  \end{align*}
  where $f_{\max}$ denotes the unique maximum of the function $f$.
  
  Plugging in this estimate with $s= t-\zeta(\omega)$ we get
  \begin{align*}
   & \Proba_x^{y}\left(\Contsymb^{y}_x\left([0,t],M\right) 
      \setminus \Contsymb^{y}_x\left([0,t],U\right)\right) \\
    \le & C \frac{f_{\max}}{f(T)} t^{-\frac{n}{2}} \e^{-d(\partial U,y)^2/(ct)} \mathbb{P}_x(\zeta \le t) \\
    \le & C' t^{-\frac{n}{2}} \e^{-d(\partial U,y)^2/(ct)}
  \end{align*}
  The constant $f_{\max}/f(T)$ depends on $y$ but if one restricts to values of $y$ such that $d(\partial U, y)$ is bounded away from zero one can pick a universal constant $C'$ that works for all such $y$. 
  
  By symmetry we can get the same estimate with $d(\partial U,x)$.
\end{proof}

Recall again that two (local) operators $D$ and $\oS D$ agree on some
subsets if there is a measure preserving local isometry intertwining
$D$ and $\oS D$ (see \Defenum{forms.ops.agree}{agree.b}).

We now state our second main result:
\begin{maintheorem}[Locality of the heat kernel]
  \label{mthm:locality_heat_kernel}
  Let $(M,d,\mu)$ and $(\oS M, \oS d, \oS \mu)$ be two metric measure
  spaces with non-negative self-adjoint operators $D$ and $\oS D$.
  Assume $D$ and $\oS D$ agree on some open subsets $U$ and $\oS U$
  via a measure preserving isometry $\map \psi U {\oS U}$.  Assume in
  addition that the associated heat kernels $p_t$ and $\oS p_t$ each
  satisfy an exponential decay bound as stated in \Def{heat_kernel_decay}. 
  Let $V$ be open with $\clo V \subset U$ and let $x,y \in V$. Then
  \begin{align*}
    \abs{p_t(x,y) - \oS p_t(\psi(x), \psi(y))} 
    \le C\e^{- \eps/t}
  \end{align*}
  for $\mu$-almost all $x,y \in V$ and all $t\in (0,T]$, where the constants $C$
  and $\eps$ depend only on $U,V$ and $T$, but not on $x$, $y$ or $t$.
\end{maintheorem}
\begin{proof}
  We can write the heat kernel $p_t(x,y)$ with the help of the Wiener
  measure and separate the set of paths from $x$ to $y$ into the local
  part that stays in $U$ and the part that leaves $U$ as in
  \Lem{heat_kernel_decomposition}. As $\psi$ is an isometry, the set
  of local paths in $U$ is the same as the local paths in $\oS{U}$. By
  \Mthm{locality_wiener_measure} the Wiener measures on these sets are
  also identical. Hence the heat kernels $p_t$ and $\oS{p}_t$ differ
  only by the Wiener measures of the non-local paths. Let $\bar \rho:=
  \inf (d(\bd V, \bd U))$. Then $\bar\rho>0$ because we assumed that
  the closure of $V$ is contained in $U$ and $\bar \rho$ is the
  infimum over the $(x,y)$-dependent $\rho$ in
  \Lem{nonlocal_wiener_measure_bound} taken over all $x,y \in
  V$. Hence if we apply \Lem{nonlocal_wiener_measure_bound} we get the
  estimate
  \begin{align*}
    \abs{p_t(x,y) - \oS p_t(\psi(x), \psi(y))} 
    \le 2Ct^{-\frac{n}{2}}\e^{- \bar{\rho}^2/ct}
  \end{align*}
  
  One can remove the $t^{-n}$ term by using the following elementary
  estimate. For all $\alpha >0$ and all $0 < b <a$ there exists a
  $C>0$ such that
  \begin{align*}
    t^{-\alpha}\e^{-a/t} < C \e^{-b/t}  
  \end{align*}
  holds for all $t>0$. 
\end{proof}

\begin{corollary}
  \label{cor:main2}
  Under the assumptions of \Mthm{locality_heat_kernel}, the asymptotic
  expansions for $p_t$ and $\oS p_t$ are identical over $V$, i.e.,
  \begin{align*}
    \Bigabs{
      \int_V p_t(x,x) \dd \mu(x)
      - \int_{\psi(V)}\oS p_t(\oS x,\oS x)\dd \oS \mu(\oS x)}
    \to_{t \to 0^+} 0 + \tilde{C}e^{-\eps / t}.
  \end{align*}
\end{corollary}

\begin{remark}
  We believe that \Mthm{locality_heat_kernel} is sharp in the sense
  that some exponential decay bound on the heat kernel is needed for
  locality of the heat kernel to hold. Grigoryan~\cite{Grigoryan03}
  considers heat kernel estimates for very general metric spaces. He
  shows that some fractals satisfy heat kernel estimates of the form
  \begin{equation*}
    p_t(x,y) 
    \le Ct^{-c_1}\exp\Bigl(-\frac{d(x,y)^{c_2}}{t^{c_3}}\Bigr)
  \end{equation*}
  for some suitable constants. One can probably extend
  \Lem{nonlocal_wiener_measure_bound} and \Mthm{locality_heat_kernel}
  to this setting. 
  
  However, he also shows that the heat kernel for the
  operator $(-\partial_x^2-\partial_y^2)^{\frac{1}{2}}$ on subsets of
  $\R^2$ with reasonably nice boundary satisfies a non-exponential
  decay bound of the form $1/C(t^2+td(x,y)) \le p_t(x,y) \le C/(t^2+td(x,y))$ and one can easily show that these heat kernels do not satisfy locality in the sense of
  \Mthm{locality_heat_kernel}.
\end{remark}

%
\section{Manifold-like spaces}
\label{sec:the.space}
%

In this section, we will define manifold-like spaces. They provide a rich class of examples where the conditions for locality can be explicitly checked and proven. We start with
metric measure spaces which satisfy the measure contraction property
(MCP), a concept first introduced in~\cite{Sturm98}. Then we define a
manifold-like space as a quotient of an MCP space with only a finite
number of points in each equivalence class being identified (and some
other conditions).
\subsection{The measure contraction property}

We need a few more notions from metric geometry.  For details we refer
to the book~\cite{BBI01}.  Let $(M,d)$ be a metric space and $\gamma$
a path in $M$, i.e., a continuous map $\map \gamma {[a,b]} M$ with
$a<b$.  For a finite number of points $T:=\{t_0,\dots,t_N\}$ with
$t_0=a<t_1<\dots<t_N=b$ let
\begin{equation*}
  L_d(\gamma,T):=\sum_{j=1}^N d(\gamma(t_{j-1}),\gamma(t_j)).
\end{equation*}
The \emph{length} $L_d(\gamma)$ of $\gamma$ is defined as the supremum
of $L(\gamma,T)$ over all partitions $T$ of $[a,b]$. The path $\gamma$
is called \emph{rectifiable} if $L_d(\gamma)$ is finite.

For a 
subset $M_0 \subseteq M$ we define the \emph{intrinsic metric
  $d_{M_0}(x,y)$ of $M_0$ in $M$} as the infimum of $L_d(\gamma)$ over
all rectifiable paths $\gamma$ from $x$ to $y$ which stay entirely in
$M_0$.  We say that $M_0$ is \emph{geodesically complete} if for all
points $x,y \in M_0$, the intrinsic metric $d_{M_0}(x,y)$ is achieved
by a shortest path $\gamma$ joining $x$ and $y$ in $M_0$, i.e., if
$d_{M_0}(x,y)=L_d(\gamma)$.  We say that $M_0$ is \emph{(geodesically)
  convex in $M$} if $M_0$ is geodesically complete and if $d_{M_0} =d
\restr{M_0 \times M_0}$, i.e., all pairs of points $(x,y) \in M_0
\times M_0$ are joined by a geodesic $\gamma$ in $M_0$ with length
given by the original metric $d$, i.e., with $L_d(\gamma)=d(x,y)$. If
$M$ is geodesically convex in itself, we say $M$ is a \emph{geodesic
  space}.  We say that $M_0$ is \emph{(geodesically) strictly convex
  (in $M$)} if the geodesic joining any pair of points is unique.

Let $B_r(x):=\set{y \in M}{d(x,y) \le r} \subset M$ denote the
(closed) ball of radius $r$ around $x$ and let $B_r^*(x)$ denote the
ball without the point $x$. Let $\LipCont M$ denote the Lipschitz
continuous functions on $M$.  For $t \in (0,1)$, a point $z$ is
\emph{$t$-intermediate} between $x$ and $y$ if $d(x,z)=td(x,y)$ and
$d(y,z)=(1-t)d(x,y)$. If $M$ is geodesically strictly convex, the
$t$-intermediate point between $x$ and $y$ is unique but in general
there can be multiple $t$-intermediate points between $x$ and $y$.

For $N=1$ set $\zeta_{K,1}^{(t)}(\theta)=t$. For $N>1$ and $K<0$ define
\begin{equation*}
  \zeta_{K,N}^{(t)}(\theta)
  := t \left(\frac  {\sinh(t\theta\sqrt{-K/(N-1)})}
    {\sinh(\theta\sqrt{-K/(N-1)})}
  \right)^{N-1}
\end{equation*}
This function defines a reference constant which represents the ratio
of volumes of the radius $t\theta$ ball to the radius $\theta$ ball in
the constant curvature $K$ space of dimension $N$. One can make
suitable adjustments for $K=0$ or $K>0$.

\begin{definition}
  A \emph{Markov kernel} from $(\Omega_1, \mathcal{F}_1)$ to
  $(\Omega_2, \mathcal{F}_2)$ with $\Omega_i$ being measurable spaces
  and $\mathcal{F}_i$ the $\sigma$-algebras of measurable sets, is a
  map $P$ that associates to each $x \in \Omega_1$ a probability
  measure $P(x,\cdot)$ on $\Omega_2$ such that for any $B \in
  \mathcal{F}_2$ the map $x \mapsto P(x,B)$ is
  $\mathcal{F}_1$-measurable.
\end{definition}

\begin{definition}[\cite{Sturm06b}]
  Let $N \ge 1$ and $K \in \R$.  A metric measure space $(M,d,\mu)$
  satisfies the \emph{$(K,N)$ measure contraction property} or
  \emph{$(K,N)$-MCP} for short if for every $t \in (0,1)$ there exists
  a Markov kernel $P_t$ from $M \times M$ to $M$ such that
  \begin{enumerate}
  \item $P_t(x,y;dz)= \delta_{\gamma_t(x,y)}(dz)$ with $\gamma_t(x,y)$
    a $t$-intermediate point between $x$ and $y$ holds for
    $\mu^2$-almost all $(x,y) \in M \times M$.
  \item for $\mu$-almost every $x \in M$ and every measurable $B
    \subseteq M$ we have
    \begin{align*}
      \int_M \zeta_{K,N}^{(t)}(d(x,y)) P_t(x,y;B) \dd\mu(y) \le \mu(B) \\
      \int_M \zeta_{K,N}^{(1-t)}(d(x,y)) P_t(y,x;B) \dd\mu(y) \le \mu(B) 
    \end{align*}
  \end{enumerate}
  
As written, this definition implies that $M$ is connected. By a slight abuse of notation we will also include disconnected spaces provided they have at most finitely many connected components and satisfy the measure contraction property on each component. 
\end{definition}

\begin{remark}
  This definition can be interpreted as a way to generalize the notion
  of a lower Ricci curvature bound and an upper dimensional bound on a
  Riemannian manifold. A Riemannian manifold with Ricci curvature at
  least $K$ and dimension $N$ satisfies the $(K,N)$-MCP.
\end{remark}

For a list of classes of spaces that satisfy this property and a few
more explicit examples see \Subsec{examples} below.

\begin{lemma}
  If $(M,d,\mu)$ satisfies the $(K,N)$-MCP then (each connected component of) $M$ is a geodesic
  space.
\end{lemma}
\begin{proof}
  The definition of the $(K,N)$-MCP implies that for $\mu \otimes
  \mu$-almost all points $(x,y) \in M \times M$ and all $t \in (0,1)$
  a $t$-intermediate point exists.  As we have assumed that $M$ is
  complete, we can replace `$\mu \otimes \mu$-almost all $(x,y)$' by
  `all $(x,y)$'.  Existence of $t$-intermediate points for all $t$
  and all $x,y$ is equivalent to being a geodesic space
  by~\cite{Sturm06a}.
\end{proof}

\begin{theorem}[see~\cite{Sturm06b}]
  \label{thm:mcp2}
  Assume the metric measure space $(M,d,\mu)$ satisfies the
  $(K,N)$-MCP for some $K \in \R$ and some $N\ge 1$.
  Then
  \begin{enumerate}
  \item 
    \label{mcp2.d}
    $(M,d,\mu)$ also satisfies the $(K',N')$-MCP for
    any $K'\le K$ and any $N'\ge N$.
  \item 
    \label{mcp2.e}
    If $M' \subseteq M$ is convex, then $(M', d \restr{M' \times
    M'}, \mu\restr{M'})$ also satisfies the $(K,N)$-MCP.
  \item
    \label{mcp2.a}
    $(M,d,\mu)$ has Hausdorff dimension at most $N$.
  \item
    \label{mcp2.b}
    For every $x \in M$ the function $r \mapsto \mu(B_r(x))/r^N$ is
    bounded away from zero for $r \in (0,1]$.
  \item
    \label{mcp2.c}
    $M$ satisfies the volume doubling property, that is there exists a
    constant $v_M$ such that for all $r>0$ and all $x \in M$ we have
    \begin{align*}
      \mu(B_{2r}(x)) \le v_M \mu(B_r(x))
    \end{align*}
  \end{enumerate}
  Note that property~\itemref{mcp2.c} follows from property~\itemref{mcp2.b}.
\end{theorem}

\begin{assumption}
  \label{ass:uniform_dimension}
  We assume from now on the following:
  \begin{enumerate}
  \item
    \label{unif_dim.a}
    The number $N$ is the exact Hausdorff dimension of $M$, that
    is the $N$ in the $(K,N)$-MCP is sharp.
 
  \item
    \label{unif_dim.b}
    The space $M$ is \emph{$N$-Alfohrs-regular}, i.e., there exists a constant
    $c>0$ such that
    \begin{equation}
      \label{eq:alfohrs-reg}
      \frac1c r^N \le \mu(B_r(x)) \le c r^N
    \end{equation}
    for all $x \in M$ and all $r \le 1$.
  \end{enumerate}
 \end{assumption}
\begin{lemma}
    Assume the metric measure space $(M,d,\mu)$ satisfies the
  $(K,N)$-MCP for some $K \in \R$ and some $N\ge 1$.
    Then the limit
    \begin{equation*}
      \tau(x):=\lim_{r\to 0} \frac{\mu(B_r(x))}{r^{N}}
    \end{equation*}
    exists for all $x \in M$. 
    
    Note that the limit function $\map \tau M {(0,\infty)}$ is in general not continuous, but globally bounded.
\end{lemma}
\begin{proof}
MCP spaces satisfy the Bishop-Gromov inequality by \cite{Sturm06b}, so $\frac{\mu(B_r(x))}{r^{N}}$ is increasing as $r \ra 0$, as it's bounded this implies the limit exists. 
\end{proof}

\begin{remark}
 These assumptions restrict the class of examples compared to~\cite{Sturm06b} but we feel we mostly excluded some pathological cases. We will call a space that satisfies the $(K,N)$-MCP \emph{and} these assumptions an \emph{MCP space}.
\end{remark}

\subsection{Glueing and manifold-like spaces}

The class of MCP spaces is already fairly large but it does not
contain some of the examples we want to study. We will extend this
class by introducing a glueing
operation.

\begin{definition}
  \label{def:identif}
  Let $(M,d,\mu)$ be a metric measure space.  We say that $\bar M$ is
  obtained from $M$ by \emph{glueing}
  \begin{itemize}
  \item if there are closed subsets $F_1$ and $F_2$ of $M$ such that
    there is a measure preserving isometry $\map \phi {F_1} {F_2}$
    and
  \item if $\bar M:=M/{\sim}$, where $\sim$ is the equivalence
    relation defined by $x \sim \phi(x)$;
  \item we assume that there exists a $k \in \N$ such that each
    equivalence class contains at most $k$ elements.
  \end{itemize}
\end{definition}
Denote the natural projection map by $\map \pi M {\bar M}$.  This
projection defines a metric measure space $(\bar M, \bar
d,\bar \mu)$ as follows: The induced distance is given by $\bar
d(\bar{x},\bar{y}):= \min \set{d(x,y)}{\pi(x)=\bar x,\pi(y)=\bar y}$,
and the induced measure is the the push forward measure $\bar
\mu:=\pi^*\mu$ (i.e., $\bar \mu(\bar{B}):= \mu(\pi^{-1}(\bar B))$).

Note that this construction includes both the possibility of glueing a metric space to itself as well as the possibility of glueing together two components of a disconnected metric space.

\begin{definition}
  \label{def:mfd-like}
  A \emph{manifold-like space} is a connected metric measure space $(\bar M,
  \bar d,\bar \mu)$ that is obtained from a (possibly not connected) MCP space $(M,d,\mu)$
  through a finite number of glueings.
\end{definition}

Note that $\bar M=M/{\sim}$ where $x \sim y$ if and only if there is a finite
sequence of isometries $\phi_1,\dots,\phi_r$ defining the glueing such
that $x=x_0$, $x_1=\phi_1(x_0)$, \dots, $y=x_r=\phi_r(x_{r-1})$.  We
still write $\map \pi M {\bar M}$ for a manifold-like space.

\begin{remark}
  Note that glueing does \emph{not} preserve the $(K,N)$-MCP property,
  as we will see in the example \Subsec{examples}.
\end{remark}

\begin{theorem}
  \label{thm:prop.identif.space}
  Let $(\bar M, \bar d,\bar \mu)$ be a manifold-like space obtained
  from the $(K,N)$-MCP space $(M,d,\mu)$. Then $\bar M$ inherits the
  following properties.
  \begin{enumerate}
  \item
    \label{prop.identif.space.a}
    $(\bar M, \bar d,\bar \mu)$ has Hausdorff dimension $N$ and is
    $N$-Alfohrs-regular, see~\eqref{eq:alfohrs-reg}.
  \item
    \label{prop.identif.space.b}
    $(\bar M, \bar d,\bar \mu)$ satisfies the volume doubling
    property. There exists a constant $v_{\bar M}$ such that for all
    $r>0$ and all $\bar{x} \in \bar M$ we have
    \begin{align*}
      \bar \mu(B_{2r}(\bar{x})) \le v_{\bar M} \bar \mu(B_r(\bar{x}))
    \end{align*}
  \item
    \label{prop.identif.space.c}
    The limit 
    \begin{align*}
      \bar{\tau}(\bar{x}) := \lim_{r\to 0} \frac{\bar \mu(B_r(\bar{x}))}{r^{N}} 
    \end{align*}
    exists for every $\bar{x} \in \bar{M}$. The limit function
    $\bar{x} \mapsto \bar{\tau}(\bar{x})$ is globally bounded on $\bar
    M$.
  \end{enumerate}
\end{theorem}
\begin{proof}
  It is clearly sufficient to prove this for one glueing. By
  definition each point in $\bar M$ has only finitely many preimages
  in $M$ under the projection map $\pi$. This shows that $\bar M$
  still has Hausdorff dimension $N$ and is $N$-Alfohrs-regular.

  As $\bar \mu$ is just the push forward metric of $\mu$,
  properties~\itemref{prop.identif.space.b}
  and~\itemref{prop.identif.space.c} are directly inherited from $M$.
\end{proof}

\subsection{Examples}
\label{ssec:examples}

In this section we will show that various classes of spaces are MCP
spaces or manifold-like. We will also exhibit a few concrete examples
and counter examples.

\begin{lemma}[\cite{Sturm06b}]
  If $(M,d,\mu)$ is a metric measure space with Hausdorff dimension $N$
  and with Alexandrov curvature bounded from below by $\kappa$, then
  $M$ satisfies the $((N-1)\kappa, N)$-MCP.
\end{lemma}

\begin{corollary}
  Compact Riemannian manifolds without boundary or with smooth
  boundary are MCP spaces.
\end{corollary}
\begin{proof}
  Compact $N$-dimensional manifolds have Alexandrov curvature bounded
  from below by the Cartan-Alexandrov-Toponogov triangle comparison
  theorem and have Hausdorff dimension $N$.
\end{proof}

A closed subset $D \subset \R^n$ is called a \emph{special Lip\-schitz
  domain} if there is a Lip\-schitz-continuous function $\map
\psi{\R^{n-1}}\R$ such that
\begin{equation*}
  D=\set{(x',x_n) \in \R^n}{\psi(x') \le  x_n \; \forall x' \in \R^{n-1}}.
\end{equation*}

\begin{definition}[cf.~\cite{MitreaTaylor99}]
  \label{def:lip.mfd}
  A pair of a compact metric measure space $(M,d,\mu)$ and a smooth
  Riemannian manifold without boundary $(\wt M, \wt g)$ is a
  \emph{smooth manifold with Lip\-schitz boundary} if the following
  holds.
  \begin{itemize}
  \item $M \subseteq \wt M$; 
  \item the metric $d$ is the metric defined via the Riemannian metric
    $\wt g$ restricted to $M$;
  \item the measure $\mu$ is the Riemannian measure defined via $\wt
    g$ restricted to $M$;
  \item $M$ and $\wt M$ have the same dimension;
  \item in the charts of $\wt M$, the boundary of $M$ in $\wt M$ is
    Lip\-schitz, i.e, for each point $p \in \bd M$ there is an open
    neighbourhood $U$ and a (smooth) chart $\map \phi U {V \subset
      \R^n}$ of the manifold $\wt M$ and a special Lip\-schitz domain
    $D \subset \R^n$ such that $M \cap U=\phi(V \cap D)$.
  \end{itemize}
\end{definition}

\begin{corollary}
  If $(M,\wt M)$ is a manifold with Lipschitz boundary and $M$ is
  convex in $\wt M$, then $M$ is an MCP space.
\end{corollary}
\begin{proof}
  The $(K,N)$-MCP property is inherited on convex subsets by
  \Thmenum{mcp2}{mcp2.e}. The dimension \Ass{uniform_dimension} is
  inherited for subsets with Lipschitz boundary. Note that Lipschitz
  continuity is crucial here. If there are cusps, this assumption may
  fail, see \Exenum{mcp}{mcp.e} below.
\end{proof}

\begin{lemma}
  Compact metric graphs are manifold-like spaces.
\end{lemma}
\begin{proof}
  A compact metric graph can be obtained from a finite number of
  finite intervals, that is manifolds with boundary, through glueing
  of the end points. 
\end{proof}
Note that any vertex of degree at least $3$ has
  Alexandrov curvature $-\infty$.

\begin{example}
  A compact good orbifold is a manifold-like space.  See~\cite{DGGW08}
  for the exact definition and a general introduction to orbifolds. A
  good orbifold is the orbit space of an isometric action by a
  discrete group on a manifold. In other words it can be obtained
  through glueing from a manifold.
\end{example}

\begin{lemma}[\cite{Ohta07}] 
  If $(M_1,d_1,\mu_1)$ and $(M_2,d_2,\mu_2)$ satisfy the
  $(K_1,N_1)$-MCP and $(K_2,N_2)$-MCP respectively, then $(M_1\times
  M_2, d_1+d_2, \mu_1 \times \mu_2)$ satisfies the $(\min(K_1,K_2),
  N_1+N_2)$-MCP. In other words, the MCP property is preserved under
  products.
\end{lemma}

\begin{example}
  Let $(M,d,\mu)$ be a $(K,N)$-MCP space. Then $M \times M$ is a
  $(K,2N)$-MCP space. This can be seen as a physical model of the
  state space of two distinguishable particles.
 
  Let $\map \phi {M \times M}{M \times M}$ be the isometry given by
  $\phi(x,y)=(y,x)$. Then $(M \times M)/{\sim}$ with ${(x,y) \sim
    \phi(x,y)}$ is a manifold-like space. This corresponds to the
  state space of two indistinguishable particles. The same
  construction applies to multi-particle systems.
\end{example}

\begin{examples}
  \label{ex:mcp}
  Some concrete examples and counter-examples:
  \begin{enumerate}
  \item
    \label{mcp.a}
    If $M$ is a flat cone (i.e. a wedge like segment of the unit
    disk in $\R^2$ with boundaries identified) it satisfies the
    $(0,2)$-MCP and is an MCP space. The Alexandrov curvature is
    $+\infty$ at the cone point and zero elsewhere.
  \item
    \label{mcp.b}
    Let $M$ be constructed as follows. Cut open the unit disk in
    $\R^2$ along the negative $x$-axis and glue in another quarter of
    the unit disk. $M$ is a pseudo cone with angle $5\pi/2$. It has
    Alexandrov curvature $-\infty$ at the cone point and does
    \emph{not} satisfy the $(K,N)$-MCP for any $K,N$
    (see~\cite{Sturm06b}) but $M$ is a manifold-like space (glued out
    of $3$ pieces to make the Lipschitz domains convex).

  \item
    \label{mcp.c}
    Let $M$ consist of two copies of the unit disk in $\R^2$ glued
    together at the origin. This is a manifold-like space but does not
    satisfy the $(K,N)$-MCP.

  \item
    \label{mcp.d}
    Let $M$ be the set of points in $\R^3$ that is the union of
    $\set{(x,y,z)}{x^2+y^2 \le 1, z=0}$ and $\set{(x,y,z)}{x^2+z^2 \le
      1, x,z \ge 0, y=0}$.  This is a manifold-like space.

  \item
    \label{mcp.e}
    Let $M:=\set{(x,y) \in [0,1]^2}{y \le x^2}$. Then the $\eps$-balls
    around $(0,0)$ have volume proportional to $\eps^3$. Hence the
    dimension \Ass{uniform_dimension} is \emph{not} satisfied and this is neither
    an MCP space nor a manifold-like space.

  \item
    \label{mcp.f}
    Let $M = [0,1]\times [0,1]/{\sim}$ where $(x,0) \sim (0,0)$ for
    all $x \in [0,1]$. Then the $\eps$-balls around $(0,0)$ have
    volume proportional to $\eps$. Hence the dimension \Ass{uniform_dimension} is
    \emph{not} satisfied and this is neither an MCP space nor a
    manifold-like space.
  \end{enumerate}
\end{examples}

%
\section{The natural Dirichlet forms}
\label{sec:dir.forms}
%

\subsection{Definition of the natural Dirichlet form}
\label{ssec:dir.mcp}
There exists a natural Dirichlet form on MCP spaces and it induces a
Dirichlet form on manifold-like spaces.

\begin{definition}[\cite{Sturm06b}]
  \label{def:dir.form1}
  Let $(M,d,\mu)$ be a metric measure space. Let
  \begin{align*}
    \qf E_r(f) := \int_M \frac{N}{r^N} \int_{B_r^*(x)}
    \left(\frac{f(y)-f(x)}{d(y,x)}\right)^2 \dd \mu(y) \dd\mu(x)
  \end{align*}
  for all $f \in \LipCont M$.
\end{definition}

\begin{theorem}[see~\cite{Sturm06b}]
  \label{thm:dirichlet_form_existence}
  Assume $(M,d,\mu)$ is an MCP space. Then the limit 
  \begin{align*}
   \qf E(f):= \lim_{r \to 0} \qf E_r(f)
  \end{align*}
exists for all $f \in \LipCont
  M$. Furthermore the closure of $\qf E$ is a regular strongly local
  Dirichlet form on $(M,d,\mu)$ with core $\LipCont M$.
\end{theorem}

\begin{remark}
  This and other theorems quoted from~\cite{Sturm06b} also hold for
  non-compact MCP spaces. In this case one needs to replace the
  function spaces with the compactly supported versions.
\end{remark}


We now define a Dirichlet form $\bar{\qf E}$ on the quotient $\bar M$
from our Dirichlet form $\qf E$ on the original space $M$ via $\map
\pi M {\bar M=M/{\sim}}$.  We can see $\bar {\qf E}$ as a
restriction of the form $\qf E$ (see remark below):
\begin{proposition}
  \label{prp:dirichlet_form_on_manifold-like}
  Let $(\bar M, \bar d,\bar \mu)$ be a manifold-like space obtained
  from the MCP space $(M,d,\mu)$.  Then the Dirichlet form $\qf E$ on
  $M$ induces a Dirichlet form $\bar{\qf E}$ on $\bar M$ as a pull
  back. This form $\bar{\qf E}$ is also regular, strongly local and
  has core $\LipCont {\bar M}$.
\end{proposition}
\begin{proof}
  As $\bar \mu$ is the push forward measure of $\mu$, the map
  $\map{\pi^*}{\Lsqr {\bar M}}{\Lsqr M}$, $\pi^* \bar f
  := \bar f \circ \pi$, is an isometry onto its image.
  The image of $\LipCont{\bar M} \subset \Lsqr{\bar M}$ under $\pi^*$
  is given by
  \begin{equation*}
    \pi^*(\LipCont {\bar M})
    = \bigset{f \in \LipCont M}{\text{$f(x)=f(y)$ whenever $x \sim y$}}
    \subset \dom \qf E
  \end{equation*}
  Hence for $\bar f \in \LipCont{\bar M}$ we define $\bar{\qf E}(\bar
  f):= \qf E (\pi^*\bar f)$. We then define $\bar{\qf E}$ to be the
  closure of this form with respect to the norm given by $\normsqr[\qf
  {\bar E}]\cdot =\normsqr[\Lsqr{\bar M}] \cdot + \bar{\qf E}(\cdot)$.

  The unit contraction property is inherited from $\qf E$, i.e. $\bar
  f \in \dom \bar{\qf E}$ implies that $\bar f^\# \in \dom \bar{\qf
    E}$ and $\bar{\qf E}(\bar f^{\#}) \le \bar{\qf E} (\bar
  f)$. Similarly, locality is inherited from $\qf E$.
  
  For the regularity of $\bar{\qf E}$, we note first that
  $\LipCont{\bar M} \subset \Cont {\bar M} \cap \dom \bar{\qf E}$ by
  definition. Hence $\Cont {\bar M} \cap \dom \bar{\qf E}$ is dense in
  $\dom \bar{\qf E}$. By Stone-Weierstrass, $\LipCont{\bar M}$ is
  dense in $\Cont{\bar M}$ in the supremum norm. Thus $\Cont {\bar M}
  \cap \dom \bar{\qf E}$ in also dense in $\Cont {\bar M}$ in the
  supremum norm.
\end{proof}

\begin{remark}
  By definition, $\map{\pi^*}{\dom \bar {\qf E}} {\dom \qf E}$
  (endowed with the natural norms $\norm[\bar {\qf E}] \cdot$ and
  $\norm[\qf E] \cdot$) is also an isometry onto its image (as
  \begin{equation*}
    \normsqr[\qf E] {\pi^* \bar f}
    = \normsqr[\Lsqr M]{\pi^*\bar f} + \qf E(\pi^* \bar f)
    =  \normsqr[\Lsqr{\bar M}] {\bar f} + \bar{\qf E}(\bar f)
  \end{equation*}
  for $\bar f$ in the core $\LipCont{\bar M}$). Hence, we can also 
  work with the corresponding image form $\dbar{\qf E}$ on $\Lsqr M$,
  which is the restriction $\dbar {\qf E} := \qf E \restr{\dom
    \dbar{\qf E}}$ of $\qf E$ with domain given by
  \begin{equation*}
    \dom \dbar{\qf E} 
    = \clo[{\norm[\qf E]\cdot}]{\bigset{f \in \LipCont M}
    {\text{$f(x)=f(y)$ whenever $x \sim y$}}}
    \subset \dom \qf E.
  \end{equation*}
  
  Note that it can happen that $\dom \dbar{\qf E} = \dom \qf E$
  although $\pi^* \LipCont{\bar M} \subsetneq \LipCont M$. This
  happens because the $\norm[\qf E]\cdot$-norm cannot see subsets of
  codimension at least two (see e.g.~\cite{ChavelFeldman78}). This
  effect can be seen in \Exenum{mcp}{mcp.c}). The Dirichlet form of
  two copies of the unit disk identified at a point is the same as the
  Dirichlet form on two disjoint copies.
\end{remark}

\subsection{Local isometries on manifold-like spaces}
\label{ssec:local.isom.manifold_like_spaces}

\begin{proposition}
  \label{prp:isomorphism_induces_form_equivalence}
  \indent
  \begin{enumiii}
  \item
    \label{iso_ind.a}
    Let $(M,d,\mu)$ and $(\oS M, \oS d, \oS \mu)$ be two MCP
    spaces with associated Dirichlet forms $\qf E$ and $\oS{\qf E}$ as
    constructed in \Thm{dirichlet_form_existence}.  If there is a
    measure preserving isometry $\map \psi U {\oS U=\psi(U)}$ for open
    subsets $U \subset M$ and $\oS U \subset \oS M$, then $\qf E$ and
    $\oS{\qf E}$ agree on $U$ and $\oS U$.

  \item
    \label{iso_ind.b}
    Let $(\bar M,\bar d, \bar \mu)$ and $(\bar{\oS M},\bar{\oS d},
    \bar{\oS \mu})$ be two manifold-like spaces with associated
    Dirichlet forms $\bar{\qf E}$ and $\bar{\oS{\qf E}}$ as constructed in
    \Prp{dirichlet_form_on_manifold-like}.  Assume that $\map \pi M
    {\bar M}$ and $\map{\oS \pi}{\oS M}{\oS {\bar M}}$ are the
    corresponding projections from MCP spaces $M$ and $\oS M$,
    respectively (see \Def{mfd-like}).
    
    If there is a measure preserving isometry $\map {\bar \psi} {\bar
      U} {\bar{\oS U}=\bar \psi(\bar U)}$ for open subsets $\bar U
    \subset \bar M$ and $\bar{\oS U} \subset \bar{\oS M}$ that lifts
    to a measure preserving isometry $\map \psi U {\oS U}$ with
    $U=\pi^{-1}(\bar U)$ and $\oS U=(\oS \pi)^{-1}(\bar{\oS U})$
    (i.e., $\bar \psi \circ \pi = \oS \pi \circ \psi$), then $\bar{\qf
      E}$ and $\bar{\oS{\qf E}}$ agree on $U$ and $\oS U$.
  \end{enumiii}
\end{proposition}
\begin{proof}
  \itemref{iso_ind.a}~The Dirichlet form on the MCP space $(M,d,\mu)$
  is defined by $\qf E(f)= \lim_{r \to 0}\qf E_r(f)$ and similarly for
  $(\oS M, \oS d, \oS \mu)$.  As $\qf E_r$ and $\oS{\qf E}_r$ are
  expressed entirely in terms of the metric $d$ and the measure $\mu$,
  we have $\qf E_r(f)=\oS {\qf E}_r(\psi_* f)$ for $f \in \LipCont M$
  with $\supp f \subset U$ and $0<r < d(\supp f, M \setminus
  U)$.
    Passing to the limit $r \to 0$ yields the
  first result. 

  \itemref{iso_ind.b}~By part~\itemref{iso_ind.a}, the lifted forms
  $\qf E$ and $\oS{\qf E}$ agree on $U$ and $\oS U$, i.e., $\qf
  E(f)=\oS{\qf E}(\psi_* f)$.  Moreover,
  \begin{equation*}
    \bar{\qf E}(\bar f)
    =\qf E(\pi^*\bar f)
    =\oS{\qf E}(\psi_* \pi^* \bar f)
    =\oS {\qf E}((\oS \pi)^*\bar \psi_* \bar f)
    =\bar{\oS {\qf E}}(\psi_* \bar f)
  \end{equation*}
  using the lift property of $\bar \psi$
  and $\psi$.
\end{proof}

\begin{corollary}
  \label{cor:isomorphism_induces_form_equivalence}
  Under the assumptions of the previous proposition, the associated
  operators on MCP resp.\ manifold-like spaces also agree.
\end{corollary}
\begin{proof}
 This follows directly from \Lem{forms.ops.agree}.
\end{proof}


\subsection{The boundary conditions for the operator}
\label{ssec:bd_cond}

We defined a Dirichlet form and the associated operator in a quite general
setting.  In this section we are going to show that for nice spaces
the operator and the Dirichlet form are very natural and familiar. 

The main motivation for the definition of the Dirichlet form
\Def{dir.form1} in~\cite{Sturm06b} is the fact that if $M$ is a
Riemannian manifold, the corresponding form is $\qf E(f) = \int_M
\abssqr{\nabla f} \dd\mu$. Additionally, his definition makes sense in
a much broader metric space setting. This statement is also true in a
local version:

\begin{proposition}
  \label{prp:mcp-mfd}
  Assume that $(M,d,\mu)$ is a manifold-like space, and $\oS M$ a boundaryless
  Riemannian manifold with its natural metric $\oS d$ and Riemannian
  measure $\oS \mu$.  If there is a measure preserving isometry $\map
  \psi U {\oS U}$ with $U \subset M$ and $\oS U \subset \oS M$ open,
  then on $U$, the form $\qf E$ just acts as $\int_{\oS U}\abssqr{\nabla
    f}\dd \oS \mu$ and the operator $D$ acts as the Laplace Beltrami operator
  on $\oS U$.
\end{proposition}
\begin{proof}
  This follows directly from
  \Prp{isomorphism_induces_form_equivalence}.
\end{proof}

\begin{definition}
  \label{def:r-fold}
  Assume that $(\bar M,\bar d,\bar \mu)$ is a manifold-like space with
  MCP lift $(M,d,\mu)$ and projection $\map \pi M {\bar M}$.  We say
  that $\bar U \subset \bar M$ is an \emph{$r$-fold smooth fibration}
  glued at a closed subset $\bar F \subset \bar M$ if the following
  holds:
  \begin{enumerate}
  \item $\bar U$ is open and connected and $\pi^{-1}(\bar
    U)=U=\bigdcup_{j=1}^r U_j$, where each $U_j$ is connected and the
    closure of each $U_j$ is isometric to a subset of a Riemannian
    manifold with smooth boundary.  The sets $U_j$ are called
    \emph{leaves}.

  \item $\bar F$ is connected and $\pi^{-1}(\bar F)=\bigdcup_{j=1}^r
    F_j$ with $F_j$ connected and $F_j \subset \bd U_j$, hence $F_j$
    is isometric to a subset of the boundary of the Riemannian
    manifold.
  \end{enumerate}
\end{definition}
To simplify notation, we assume that $U_j$ and $F_j$ are already
subsets of a Riemannian manifold (the former open in the interior, the
latter a closed subset of the boundary).  If $\map{\bar f}{\bar M}\R$
denote by $\map f M \R$ the lift of $\bar f$ onto $M$, i.e., $f \circ
\pi=\bar f$.  If $f$ is smooth enough on each $U_j$, we define
$\normder f_j$ as the normal (outward) derivative of $f$ on $\hat U_j
:= F_j \cup U_j$, and we pull back all functions $\normder f_j
\restr{F_j}$ formally defined on $F_j \subset M$ onto $\bar F$ via the
isometries and denote them by $\map {\normder \bar f_j} {\bar F}\R$.

Note that $U \setminus \bigcup_{j=1}^r F_j$ is naturally the same as
$\bar U \setminus \bar F$, as $\pi$ does not identify any points here.
 Moreover, these two
sets also have the same measure, and integrals over them agree. Therefore, we consider these sets as the same. 
\begin{proposition}
  \label{prp:kirchhoff}
  Let $(\bar M,\bar d,\bar \mu)$ be a manifold-like space obtained
  from the MCP space $(M,d,\mu)$.  Let $\bar U \subset \bar M$ be an
  $r$-fold smooth fibration with leaves $U_j$ glued at $F_j$.

  If $\bar f$ is in the domain of the associated operator $\bar D$ on
  $\bar M$ with $\supp \bar f \subset \bar U$ then $\bar f$ acts on
  $\bar U$ as the usual Laplacian ($\bar D \bar f = -\Delta \bar f$).

  Moreover, the normal derivatives on the leaves satisfy the so-called
  \emph{Kirchhoff} condition on the glued part $\bar F$. This means that
  \begin{equation*}
    \sum_{j=1}^r \normder \bar f_j =0 \qquad \textit{on } \bar F
  \end{equation*}
  and that $\bar f$ is continuous on $\bar F$.
\end{proposition}
Note that the derivatives are only weak derivatives. This theorem does
not make any statements on the regularity of $\dom \bar D$.  If we are
only on parts which are $r$-fold smooth fibrations, then the solutions
are in $\Sobspace[2]$.
\begin{proof}
  Let $\bar g \in \dom \bar{\qf E}$ with $\supp \bar g \subset \bar U$
  and $g$ its lift.  Now after our notes made above, we have
  \begin{align*}
    \int_{\bar U} \bar D \bar f \cdot \bar g \dd \bar \mu
    =& \bar{\qf E}(\bar f, \bar g)\\
    =& \sum_{j=1}^r \int_{U_j} \nabla f \cdot \nabla g \dd \mu\\
    =&\sum_{j=1}^r 
    \Bigl(
      \int_{U_j} (-\Delta f) \cdot g \dd \mu
     +\int_{F_j} \normder f_j \cdot g \dd \sigma
     \Bigr)\\
    =& \int_{\bar U \setminus \bar F} (-\Delta \bar f) \cdot g \dd \mu
     +  \int_{\bar F} \Bigl(\frac 1 r\sum_{j=1}^r\normder \bar f_j\Bigr) 
           \cdot \bar g \dd \bar \sigma   
  \end{align*}

  using Green's formula on the Riemannian manifold (third equality).
  Here, $\sigma$ denotes the canonical measure on the boundary of the
  Riemannian manifold and $\bar \sigma$ the push forward measure on
  $\bar F$ (counting each measure from the leaves boundary, hence
  $\bar \sigma (\bar F)=r\sigma(F_j)$.  We first see that $\bar D \bar
  f= -\Delta \bar f$ (choose $\bar g$ with support away from $\bar
  F$).  Then we let $\bar g \in \dom \bar{\qf E}$ with $\supp \bar g
  \subset \bar U$; as $\bar g \restr {\bar F}$ runs through a dense
  subspace of $\Lsqr{\bar F,\bar \sigma}$, the result follows.
\end{proof}

If $r=1$, this reduces to the manifold case with Neumann boundary
conditions.
\begin{corollary}
  With the same notation as above and the additional assumption that
  $r=1$, functions $\bar f \in \dom \bar D$ with $\supp \bar f \subset
  \bar U$ satisfy Neumann boundary conditions $\normder \bar f=0$ on
  $\bar F$.
\end{corollary}

\begin{examples}
  The simplest example of the situation in \Prp{kirchhoff} is a metric
  graph. The MCP space consists of a collection of disjoint intervals,
  one for each edge of the metric graph. The glueing then identifies
  the end points of the intervals that correspond to adjacent edges in
  the metric graph.
 
  \Exenum{mcp}{mcp.d} is a higher dimensional version.
\end{examples}

\subsection{Heat kernel estimates}

\begin{theorem}[\cite{CKS87}]
\label{thm:existence_measure}
  Let $(M,d,\mu)$ be a compact metric measure space and $\qf E$ a
  regular Dirichlet form on it.  Then there exists a measure
  $\Gamma(f)$ such that
  \begin{align*}
    \qf E(f) = \int_M \dd\Gamma(f)(x)
  \end{align*}
  for any $f \in \Cont M \cap \dom \qf E$.
\end{theorem}

\begin{lemma}[Subpartitioning lemma] 
  \label{lem:subpartitioning_lemma_proven}
 
  Let $(M,d,\mu)$ be a $(K,N)$-MCP space and let $U \subset M$ be open and convex. Then
  \begin{align*}
    \int_U \frac{N}{r^N} \int_{B_r^*(x) \cap U} 
       \left(\frac{f(y)-f(x)}{d(y,x)}\right)^2\dd\mu(y)\dd\mu(x) 
    \le \int_U \dd\Gamma(f)(x)
  \end{align*}
\end{lemma}
\begin{proof}
  For any MCP space and $0 = t_0 < t_1 < \dots t_{n-1} < t_n=1$ a
  partition of the unit interval we have the estimate
  \begin{equation*}
    \qf E_r(f) \le \sum_{i=1}^{n}(t_i-t_{i-1}) \qf E_{(t_i-t_{i-1})r}(f)
  \end{equation*}
  by~\cite{Sturm06b}. We will apply this directly to $U$, which is
  also a $(K,N)$-MCP space by \Thmenum{mcp2}{mcp2.e}.  Let $r_n :=
  2^{-n}r$ then $\qf E(f)= \lim_{n \to \infty} \qf E_{r_n}(f)$. Using
  the partition $0,\frac12,1$ this is an increasing sequence. Hence
  $\qf E_r(f) \le \qf E(f)$.
\end{proof}

\begin{definition}
  \label{def:e-metric}
  Let $(M,d,\mu)$ be a metric measure space and $\qf E$ a regular
  Dirichlet form on it. Then we define the \emph{energy metric}
  $\rho$ on $M$ as follows
  \begin{align*}
    \rho(x,y) := 
    \sup \Bigset{f(x) - f(y)} 
             {f \in \Cont M \cap \dom \qf E, 
               \frac{\dd \Gamma(f)}{\dd\mu} \le 1 \; \text{on}\; M}
  \end{align*}
  where $\dd\Gamma(f)/\dd \mu$ represents the Radon-Nikod\'ym
  derivative. Note that this includes the implicit assumption that
  $\dd\Gamma(f)$ is absolutely continuous with respect to $\dd\mu$.
\end{definition}

This metric is often called the intrinsic metric, especially when $M$
is only a (sufficiently nice) topological space. To avoid confusion
with the distance induced via the length of paths we use the term  energy metric.  A priori, the  energy metric need not be a proper metric, it can be degenerate.

\begin{remark}
  \label{rem:absolute_continuity}
  Let $(M,d,\mu)$ be an MCP space and let $\qf E$ be the associated
  Dirichlet form. Then $\dd\Gamma(f)$ is absolutely continuous with
  respect to $\dd\mu$ for all $f \in \Cont M \cap \dom \qf E$
  by~\cite[Cor.~6.6~(iii)]{Sturm06b}.
\end{remark}

\begin{lemma}
  \label{lem:metric_equivalence}
  Let $(M,d,\mu)$ be a $(K,N)$-MCP space. Let $\qf E$ be the
  associated Dirichlet form from \Thm{dirichlet_form_existence}.  Then
  the energy metric is equivalent to the metric $d$. In particular
  they induce the same topology on $M$.
\end{lemma}
\begin{proof}
  This is proven in~\cite{Sturm98} for a different version of the
  MCP. This proof is an adaptation of his proof.

  We have $\qf E (f) = \int_M d\Gamma(f) = \int_M
  \frac{d\Gamma(f)}{\dd\mu}(x) \dd\mu(x)$ by \Thm{existence_measure}
  and \Rem{absolute_continuity}.
  
  For $z \in M$ and $C>0$ fixed, let $f(x):=Cd(x,z)$. Then we have
  \begin{align*}
    \qf E_r(f) 
    &= \int_M \frac{N}{r^N} \int_{B_r^*(x)}C^2
       \left(\frac{d(x,z)-d(y,z)}{d(x,y)}\right)^2
       \dd\mu(y)\dd\mu(x) \\
    &\le \int_M \frac{N}{r^N} \int_{B_r^*(x)}C^2 \dd\mu(y)\dd\mu(x) .
  \end{align*}
  We assumed in~\eqref{eq:alfohrs-reg} that $\mu(B_r(x))/r^N$ is
  globally bounded by some constant $c$. Hence we can apply the
  dominated convergence theorem and get
  \begin{align*}
   \frac{d\Gamma(f)}{d\mu}(x) 
   = \lim_{r \ra 0} \frac{N}{r^N} 
       \int_{B_r^*(x)}C^2 \left(\frac{d(x,z)-d(y,z)}{d(x,y)}\right)^2
       \dd\mu(y)
  \end{align*}
  This shows $\dd \Gamma(f)/\dd \mu\le 1$ for $C \le (cN)^{-1/2}$.
 
  Plugging $f$ into the definition of the energy metric, we obtain
  $\rho(x,y) \ge C(d(x,z)-d(y,z))$ valid for any $C \le (cN)^{-1/2}$
  and any $z \in M$. In particular, for $z:=y$ we get the lower bound
  $\rho(x,y) \ge Cd(x,y)$.
 
  Let $f \in \LipCont M$ with $\dd \Gamma(f)/\dd \mu \le 1$ and let
  $L_f$ denote the sharp Lipschitz constant of $f$. Then $f(x)-f(y)
  \le L_fd(x,y)$. Hence if we show that there exists a global
  Lipschitz constant $L$ for all functions $f \in \LipCont M$ that
  satisfy $\dd\Gamma(f)/\dd\mu \le 1$ we get the estimate $\rho(x,y)
  \le Ld(x,y)$.

  Let $x_0,y_0 \in M$ be such that $f(x_0)-f(y_0)\ge
  \frac{L_f}{2}d(x_0,y_0)$. We can assume without loss of generality
  that $d_0:=d(x_0,y_0)$ is arbitrarily small by repeatedly taking
  midpoints. Hence $d_0$ can be bounded from above independent of
  $f$. Let $x \in B_{d_0/6}(x_0)$ and $y \in B_{d_0/6}(y_0)$. Then
  \begin{align*}
    |f(x)-f(y)| 
    \ge |f(x_0)-f(y_0)| - |f(x)-f(x_0)| - |f(y)-f(y_0)| \\
    \ge \left( \frac{L_f}{2} - \frac{L_f}{6} - \frac{L_f}{6}\right)d_0 
    = \frac{L_fd_0}{6} 
    \ge \frac{L_f}{12}d(x,y)
  \end{align*}

  Let $U:= B_{2d_0}(x_0)$ and $r=2d_0$ in
  \Lem{subpartitioning_lemma_proven}, then
  \begin{align*}
    \mu(B_{2d_0}(x_0)) 
    \ge& \int_{B_{2d_0}(x_0)} \dd\Gamma(f) \\
    \ge& \int_{B_{2d_0}(x_0)} \frac{N}{(2d_0)^N} 
      \int_{B_{2d_0}^*(x)\cap B_{2d_0}(x_0)} 
      \left(\frac{f(y)-f(x)}{d(y,x)}\right)^2 \dd\mu(y)\dd\mu(x) \\
    \ge& \int_{B_{d_0/6}(x_0)} \frac{N}{(2d_0)^N} \int_{B_{d_0/6}^*(y_0)} 
      \left(\frac{f(y)-f(x)}{d(y,x)}\right)^2 \dd\mu(y)\dd\mu(x) \\
    \ge & \frac{L_f^2}{144}\frac{N}{(2d_0)^N}\mu(B_{d_0/6}(x_0)) \mu(B_{d_0/6}(y_0))
  \end{align*}
  By~\eqref{eq:alfohrs-reg} we have uniform global bounds for the
  volumes of balls. All the $d_0$ cancel out, so this proves an upper
  bound for $L_f$ that is independent of $f$ completing the proof.
\end{proof}

\begin{lemma}
  Let $(\bar M,\bar d,\bar \mu)$ be a manifold-like space with induced
  Dirichlet form $\bar{\qf E}$. Then the  energy metric is
  equivalent to the $\bar d$ metric.
\end{lemma}
\begin{proof}
  It is sufficient to prove this for one glueing with glueing map
  $\phi$ and projection $\pi$.  We can write $\bar d(\bar x,\bar y)=
  \min \set{d(x,y)}{\pi(x)=\bar x,\pi(y)=\bar y}$ and similarly for
  the energy metric. Hence the equivalence of metrics is
  inherited through glueing.
\end{proof}

\begin{definition}
  Let $(M,d,\mu)$ be a metric measure space that is $N$-Alfohrs
  regular and let $\qf E$ be a regular Dirichlet form on $M$. Let
  $N^*:= \max \{3,N\}$. Then we say $\qf E$ satisfies the
  \emph{Sobolev inequality} if there exists a $C>0$ such that for all
  $f \in \dom{\qf E} \cap \Contc{B_r(x)}$ we have
  \begin{align*}
    \left(\int_{B_r(x)}\abs f^{(2N^*)(N^*-1)} \dd\mu\right)^{(N^*-2)/N^*} 
    \le C \frac{r^2}{\mu(B_r(x))^{2/N^*}} 
          \left(\int_{B_r(x)} \dd\Gamma(f) 
              + \frac 1{r^2} \int_{B_r(x)}\abssqr f \dd\mu \right) 
  \end{align*}
  for all $r>0$.
\end{definition}

\begin{theorem}[\cite{Sturm95}]
  \label{thm:heat_kernel_original} 
  Let $(M,d,\mu)$ be a metric measure space and $\qf E$ a strongly
  local regular Dirichlet form on it.

  Assume $M$ satisfies the volume doubling property, the Sobolev
  inequality and the topology induced by $\rho$ is the same as the one
  induced by $d$. Then for any $T>0$ and any $\eps>0$ there exists a
  $C>0$ such that the heat kernel estimate
  \begin{align*}
    p_t(x,y) 
    \le C \mu(B_{\sqrt{t}}(x))^{-1/2}\mu(B_{\sqrt{t}}(y))^{-1/2} 
        \exp\left(-\frac{\rho^2(x,y)}{(4+\eps)t}\right)
  \end{align*} 
  is valid for all $x,y \in M$ and all $0<t<T$.
\end{theorem}

\begin{corollary}
  \label{cor:heat_kernel_decay}
  Let $(M,d,\mu)$ be a $(K,N)$-MCP space. Let $\qf E$ be the strongly
  local regular Dirichlet form from
  \Thm{dirichlet_form_existence}. Then the heat kernel satisfies an
  exponential decay bound as in \Def{heat_kernel_decay}. That is
  \begin{align*}
    p_t(x,y) \le C t^{-N/2} \exp\left(-\frac{d^2(x,y)}{ct}\right)
  \end{align*} 
  holds for some $C,c>0$ independent of $x,y \in M$ and of $t$ as long as
  $0<t<T$ for some $T$.
\end{corollary}
\begin{proof}
  $(M,d,\mu)$ satisfies a parabolic Harnack inequality
  by~\cite{Sturm06b}. This also implies that $\qf E$ satisfies the
  Sobolev inequality~\cite{Sturm06b}. \Lem{metric_equivalence}
  gives the equivalence of the energy metric and $d$. Hence the
  assumptions of \Thm{heat_kernel_original} are satisfied.
  
  As the volume of radius $\sqrt{t}$-balls is bounded and the metrics
  $d$ and $\rho$ are equivalent, we can reformulate the bound
  from~\cite{Sturm95} as written.
\end{proof}

\begin{corollary}
  Let $(\bar M,\bar d,\bar \mu)$ be a manifold-like space with induced
  Dirichlet form $\bar{\qf E}$. Then the heat kernel satisfies the
  exponential decay bound in \Def{heat_kernel_decay}.
\end{corollary}
\begin{proof}
  We are going to use \Thm{heat_kernel_original} again. The only
  assumption that is missing is the Sobolev inequality.
 
  Let $(M,d,\mu)$ be the $(K,N)$-MCP space that $\bar M$ was obtained
  from.  The Sobolev inequality holds on $M$ by~\cite{Sturm06b}. As
  $\bar{\qf E}$ is defined as a restriction of $\qf E$ and the measure
  $\bar \mu$ on $\bar M$ is just the push forward measure, the Sobolev
  inequality also holds on $\bar M$.
\end{proof}

The manifold-like spaces we defined in \Sec{the.space} provide a large
class of examples where locality holds:
\begin{maintheorem}
  \label{mthm:manifold_like_spaces}
  Let $(M,d,\mu)$ and $(\oS M, \oS d, \oS \mu)$ be two manifold-like
  spaces. Let $\qf E$ and $\qf {\oS E}$ be the natural Dirichlet forms
  on $M$ and $\oS M$ from \Prp{isomorphism_induces_form_equivalence}.
  Let $p_t$ and $\oS p_t$ be the associated heat kernels and let
  $\Proba$ and $\oS \Proba$ the associated Wiener measures.
 
  Let $U \subset M$ be open and assume there exists a measure
  preserving isometry $\map \psi U {\oS U \subset \oS M}$.
 
  Then the Wiener measures $\Proba$ and $\oS \Proba$ are identical on
  $U$.   
  
  Let $V$ be open with $\clo V \subset U$ and let $x,y \in V$.
  Then the difference of the heat kernels is
  exponentially small, that is
  \begin{align*}
    \abs{p_t(x,y) - \oS p_t(\psi(x), \psi(y))} 
    \le C\e^{- \eps/t}
  \end{align*}
  for $\mu$-almost all $x,y \in V$ and all $t \in (0,T]$. The asymptotic expansions
  of $p_t$ and $\oS p_t$ agree on $V$.
\end{maintheorem}
\begin{proof}
  If $\psi$ is a measure preserving isometry, then the Dirichlet forms
  and the operators are equivalent by
  \Prp{isomorphism_induces_form_equivalence}. Hence we get equivalence
  of the Wiener measures by \Mthm{locality_wiener_measure}. The heat
  kernels associated to these Dirichlet forms satisfy the heat kernel
  decay bound by corollary \Cor{heat_kernel_decay}. Thus we can apply
  \Mthm{locality_heat_kernel} and get locality of the heat kernel.
\end{proof}

%
\section{Example application: a two particle system on a metric graph}
\label{sec:example}
%

Let $G=(V,E)$ be a compact metric graph with vertex set $V$ and edge
set $E$. A metric graph is a combinatorial graph together with an
assignment of edge lengths. The operator is the Laplacian, that is the
second derivative on the edges seen as intervals. We impose the
standard Kirchhoff boundary conditions at all vertices. This means
functions are continuous and the sum of the first derivatives on all
edges adjacent to a vertex is zero (the derivatives are oriented away
from the vertex). A metric graph together with the operator is called
a quantum graph. See for example~\cite{BerkolaikoKuchment13} for an
introduction and a survey of quantum graphs.
 
The manifold-like space we will look at is $M:= G \times G /{\sim}$
with $(x,y) \sim (y,x)$. This is a model from physics, it corresponds
to two particles moving freely on a metric graph. The particles do not
interact and they are indistinguishable, hence we factor out by the
symmetry.
 
This is a $2$-dimensional space which is neither a manifold nor an orbifold. To the best of our knowledge, the results in this paper are the first to explicitly show that these kind of spaces do have `well-behaved' heat kernels.
 
In order to compute the heat asymptotics of the heat kernel for this
system, we will decompose the state space of the two particles into
various pieces. This is were the locality of the heat kernel comes
in. For each piece we will explicitly compute the heat kernel of a
different space which is locally isometric to the piece of $M$ but
globally a much simpler space. For these much simpler spaces one can
write down an explicit expression of the heat kernel and use it to
compute the asymptotics.

Pick a universal $\eps >0$ much smaller than any edge length. We say
that a particle is in the neighbourhood of a vertex if it less than
$\eps$ away from it.  We decompose $M$ into the following types of
pieces.
\begin{enumABC}
\item both particles are away from vertices and on distinct edges
\item both particles are away from vertices and on the same edge
\item one particle is in a neighbourhood of a vertex, the other one is
  away from the vertices on an edge
\item the particles are in neighbourhoods of two distinct vertices
\item both particles are in the neighbourhood of a vertex
\end{enumABC}
Note that the cutoffs between the pieces need to be made in a way that
intersects the singular pieces of $M$ orthogonally otherwise these
cutoffs will produce additional terms in the asymptotics. 

For the pieces of type $A$, $C$ and $D$ the particles cannot run into
each other. So the heat kernel is just the product of the heat kernels
of the two pieces. For pieces of type $B$ and $E$ the heat kernel is
the product of the individual heat kernels modded out by the symmetry.
 
For a single particle away from vertices we can just use the real line
as a comparison space, the heat kernel is
$p_{\R}(t,x,y)=\frac{1}{\sqrt{4\pi t}}\e^{-(x-y)^2/4t}$. For a
particle in the neighbourhood of a vertex we use the star shaped
metric graph consisting of $k$ half-infinite edges all meeting in a
single central vertex as a comparison space.  Its heat kernel can be
written down explicitly. For $\alpha, \beta \in [1, \hdots, \deg(v)]$,
we write $x^{\alpha}$ when $x$ is on the edge $\alpha$. The heat
kernel is then given by 
\begin{align*}
  p_{S}(t,x^{\alpha},y^{\beta}) 
  = \frac{1}{\sqrt{4\pi t}}\left(\delta_{\alpha \beta}\e^{-(x^{\alpha}-y^{\beta})^2/4t} 
    + \sigma_{\alpha \beta}^v\e^{-(x^{\alpha}+y^{\beta})^2/4t}\right)
\end{align*}
(see e.g.~\cite{Roth83} or \cite{KPS07} for a more general version), where $\delta_{\alpha \beta}$ is the
Kronecker-$\delta$ and $\sigma_{\alpha \beta}^v$ is the matrix of the
boundary conditions at the vertex. For Kirchhoff conditions we have
$\sigma_{\alpha \beta}^v=-\delta_{\alpha \beta} + \frac{2}{\deg(v)}$.
 
We will just carry out the computations for pieces of type $C$ and
$E$. The other types work in exactly the same way.

Let $N(v)$ denote the $\eps$-neighbourhood of the vertex $v$ and
$l_{\gamma}$ the edge where the other particle is located. The
particles do not interact, hence we just multiply a heat kernel of the
star graph with a heat kernel on the real line. 

We can now apply \Mthm{manifold_like_spaces} which says that the heat kernel on $M$ in the region C and the heat kernel on the comparison space differ only by an exponentially small error term. Hence,
on the diagonal the heat kernel of the region C is given by
\begin{align*}
  p_C(t,(x_1^{\alpha},x_2),(x_1^{\alpha},x_2)) =\frac{1}{4\pi
    t}\left(1 + \sigma^v_{\alpha \alpha}\e^{-(x_1^{\alpha})^2/t}
  \right) +O(t^\infty)
\end{align*}
Here and in further computations we use the notation $O(t^{\infty})$ to mean an error term that can be bounded as $O(t^k)$ for any $k$. Such an error will make no contribution to the asymptotics we are interested in.
Thus the contribution to the asymptotics is
\begin{align*}
  &\int_C p_C(t,(x_1,x_2),(x_1,x_2))\dd x_1\dd x_2 \\
  =&\frac{1}{4\pi t}(l_{\gamma}-2\eps)\eps \deg(v) 
  + \frac{1}{4\pi t} (l_{\gamma}-2\eps)
  \sum_{\alpha \sim v}\sigma^v_{\alpha \alpha} 
  \int_{0}^{\eps} \e^{-(x_1^{\alpha})^2/t}\dd x_1^{\alpha}\\
  =&\frac{1}{4\pi t}\vol(C) + \frac{1}{8\sqrt{\pi
      t}}(l_{\gamma}-2\eps)\sum_{\alpha \sim v}\sigma^v_{\alpha
    \alpha} +O(t^\infty)
\end{align*}

To get an expression for the heat kernel on the diagonal of the region
E, we also need to know it in a neighbourhood of the diagonal. In a
neighbourhood of the diagonal we can just assume that both points on
$M$ are in the region $E$.

For two distinguishable particles, the heat kernel on the domain $N(v)
\times N(v)$ is just the product, that is
\begin{align*}
  &p_{N(v) \times N(v)}(t,(x_1^{\alpha},x_2^{\beta}),(y_1^{\gamma},y_2^{\delta})) \\
  =&\frac{1}{4\pi t} \left(\delta_{\alpha
      \gamma}\e^{-(x_1^{\alpha}-y_1^{\gamma})^2/4t} + \sigma^v_{\alpha
      \gamma}\e^{-(x_1^{\alpha}+y_1^{\gamma})^2/4t}\right)\cdot
  \left(\delta_{\beta \delta}\e^{-(x_2^{\beta}-y_2^{\delta})^2/4t} +
    \sigma^v_{\beta
      \delta}\e^{-(x_2^{\beta}+y_2^{\delta})^2/4t}\right)
 \end{align*}
 We now factor out by the isometry $\phi: (x_1,x_2) \mapsto
 (x_2,x_1)$. To get the heat kernel on $\oS{E}:=N(v) \times N(v)
 \slash_{(x_1,x_2) \sim \phi(x_1,x_2)}$ we use a result
 of~\cite{Donnelly79} which says that in our setting
 $p_{\oS{E}}(t,x,y)=p_{N(v)\times N(v)}(t,x,y)+ p_{N(v)\times
   N(v)}(t,\phi(x),y)$.  We parametrise $\oS{E}$ as follows,
 $(x_1^{\alpha},x_2^{\beta}) \in \oS{E}$ where $(x_1^{\alpha},x_2^{\beta})
 \in N(v)\times N(v)$ and either $\alpha < \beta$ or ($\alpha=\beta$
 and $x_1^{\alpha} \le x_2^{\alpha}$). Now by \Mthm{manifold_like_spaces} the heat kernel on $M$ in the region $E$ is equal to the one on $\oS{E}$ up to an exponentially small error term. Hence 
\begin{align*}
  &p_{E}(t,(x_1^{\alpha},x_2^{\beta}),(y_1^{\gamma},y_2^{\delta})) \\
  =&\frac{1}{4\pi t} \left(\delta_{\alpha
      \gamma}\e^{-(x_1^{\alpha}-y_1^{\gamma})^2/4t} + \sigma^v_{\alpha
      \gamma}\e^{-(x_1^{\alpha}+y_1^{\gamma})^2/4t}\right)\cdot
  \left(\delta_{\beta \delta}\e^{-(x_2^{\beta}-y_2^{\delta})^2/4t}
    + \sigma^v_{\beta \delta}\e^{-(x_2^{\beta}+y_2^{\delta})^2/4t}\right) \\
  &+ \frac{1}{4\pi t} \left(\delta_{\beta
      \gamma}\e^{-(x_2^{\beta}-y_1^{\gamma})^2/4t} + \sigma^v_{\beta
      \gamma}\e^{-(x_2^{\beta}+y_1^{\gamma})^2/4t}\right)\cdot
  \left(\delta_{\alpha \delta}\e^{-(x_1^{\alpha}-y_2^{\delta})^2/4t} +
    \sigma^v_{\alpha
      \delta}\e^{-(x_1^{\alpha}+y_2^{\delta})^2/4t}\right)
  +O(t^\infty)
\end{align*}
On the diagonal this gives
\begin{align*}
  &p_{E}(t,(x_1^{\alpha},x_2^{\beta}),(x_1^{\alpha},x_2^{\beta})) \\
  =&\frac{1}{4\pi t} \left(1 + \sigma^v_{\alpha
      \alpha}\e^{-(x_1^{\alpha})^2/t}\right)\cdot
  \left(1 + \sigma^v_{\beta \beta}\e^{-(x_2^{\beta})^2/t}\right) \\
  & + \frac{1}{4\pi t} \left(\delta_{\alpha
      \beta}\e^{-(x_1^{\alpha}-x_2^{\beta})^2/4t} + \sigma^v_{\alpha
      \beta}\e^{-(x_1^{\alpha}+x_2^{\beta})^2/4t}\right)^2
  +O(t^\infty)
 \end{align*}

 We will separate the integration into two parts, first the region
 where $\alpha < \beta$ and second the region where $\alpha=\beta$ and
 $x_1^{\alpha}\le x_2^{\alpha}$. In the first region we can integrate
 over the rectangle $[0,\eps]\times [0,\eps]$ for each pair of edges.
\begin{align*}
  &\sum_{\alpha < \beta}\int_0^{\eps}\int_0^{\eps}
  p_{E}(t,(x_1^{\alpha},x_2^{\beta}),(x_1^{\alpha},x_2^{\beta}))
  \dd x_1^{\alpha}\dd x_2^{\beta}\\
  =&\frac{1}{4\pi t}\sum_{\alpha < \beta}\int_0^{\eps}1 +
  \sigma^v_{\alpha \alpha}\e^{-(x_1^{\alpha})^2/t}\dd x_1^{\alpha}
  \int_0^{\eps} 1 + \sigma^v_{\beta \beta}\e^{-(x_2^{\beta})^2/t}\dd x_2^{\beta}\\
  & + \frac{1}{4\pi t}\sum_{\alpha < \beta}(\sigma^v_{\alpha
    \beta})^2\int_0^{\eps}\int_0^{\eps}
  \e^{-(x_1^{\alpha}+x_2^{\beta})^2/2t}\dd x_1^{\alpha}\dd x_2^{\beta} +O(t^\infty)\\
  =&\frac{1}{4\pi t}\sum_{\alpha < \beta} \left( \eps +
    \sigma^v_{\alpha \alpha}\frac{\sqrt{\pi}}{2}t^{1/2}\right)\cdot
  \left( \eps + \sigma^v_{\beta
      \beta}\frac{\sqrt{\pi}}{2}t^{1/2}\right)
  +\frac{1}{4\pi t}\sum_{\alpha < \beta} (\sigma^v_{\alpha \beta})^2 t  + O(t^\infty)\\
  =&\frac{1}{4\pi t}\frac{1}{2}\deg(v)(\deg(v)-1)\eps^2
  + \frac{1}{8\sqrt{\pi t}}\sum_{\alpha < \beta} 
  (\sigma^v_{\alpha \alpha}+\sigma^v_{\beta \beta})\eps \\
  & + \frac{1}{16}\sum_{\alpha < \beta}\sigma^v_{\alpha \alpha}
  \sigma^v_{\beta \beta}
  +\frac{1}{4\pi}\sum_{\alpha < \beta}(\sigma^v_{\alpha \beta})^2  + O(t^\infty)\\
\end{align*}

For the second part, where $\alpha=\beta$ we will drop the superscript
and write $x_1$ and $x_2$ for $x_1^{\alpha}$ and $x_2^{\alpha}$ to
simplify notation.
\begin{align*}
  &e_{E}(t,(x_1,x_2),(x_1,x_2)) \\
  =&\frac{1}{4\pi t} \left(1 + \sigma^v_{\alpha
      \alpha}\e^{-x_1^2/t}\right)\cdot \left(1 + \sigma^v_{\alpha
      \alpha}\e^{-x_2^2/t}\right) + \frac{1}{4\pi t}
  \left(\e^{-(x_1-x_2)^2/4t}
    + \sigma^v_{\alpha \alpha}\e^{-(x_1+x_2)^2/4t}\right)^2 +O(t^\infty)\\
  =& \frac{1}{4\pi t} + \frac{1}{4\pi t}\sigma^v_{\alpha
    \alpha}\e^{-x_1^2/t}+\frac{1}{4\pi t}\sigma^v_{\alpha
    \alpha}\e^{-x_2^2/t}
  +\frac{1}{4\pi t}(\sigma^v_{\alpha \alpha})^2\e^{-(x_1^2+x_2^2)/t}\\
  &+\frac{1}{4\pi t} \e^{-(x_1-x_2)^2/2t} +\frac{1}{2\pi
    t}\sigma^v_{\alpha \alpha}\e^{-(x_1^2+x_2^2)/2t}
  +\frac{1}{4\pi t}(\sigma^v_{\alpha \alpha})^2\e^{-(x_1+x_2)^2/2t} +O(t^\infty)\\
 \end{align*}

 In the second region we will integrate over a difference of two
 triangles to ensure that the region meets the boundary orthogonally.
 The first triangle is described by $0 \le x_2 \le x_1 \le \eps$ with
 area $\frac{1}{2}\eps^2$. The second triangle is the region where
 $2^{-1/2}\eps \le x_1 \le \eps$ and $2^{1/2}\eps-x_1 \le x_2 \le x_1$
 with area
 $\frac{1}{2}(1-2^{-1/2})^2\eps^2=(\frac{3}{8}-2^{-1/2})\eps^2$.
 Therefore
 \begin{align*}
   & 4\pi t\int_0^{\eps}\int_0^{x_1}e_{E}(t,(x_1,x_2),(x_1,x_2)) \dd x_2\dd x_1 \\
   & - 4\pi t\int_{2^{-1/2}\eps}^{\eps}\int_{2^{1/2}\eps-x_1}^{x_1}
   e_{E}(t,(x_1,x_2),(x_1,x_2)) \dd x_2\dd x_1 \\
   =& \frac{1}{2}\eps^2 + \sigma^v_{\alpha
     \alpha}\int_0^{\eps}x_1\e^{-x_1^2/t}\dd x_1
   +\sigma^v_{\alpha \alpha}\int_0^{\eps}\int_0^{x_1}\e^{-x_2^2/t}\dd x_2\dd x_1 \\
   & + (\sigma^v_{\alpha
     \alpha})^2\int_0^{\eps}\int_0^{x_1}\e^{-(x_1^2+x_2^2)/t}\dd
   x_2\dd x_1
   +\int_0^{\eps}\int_0^{x_1} \e^{-(x_1-x_2)^2/2t}  \dd x_2\dd x_1\\
   &+2\sigma^v_{\alpha
     \alpha}\int_0^{\eps}\int_0^{x_1}\e^{-(x_1^2+x_2^2)/2t} \dd x_2\dd
   x_1
   +(\sigma^v_{\alpha \alpha})^2\int_0^{\eps}\int_0^{x_1}\e^{-(x_1+x_2)^2/2t}
   \dd x_2\dd x_1 \\
   & - (\frac{3}{8}-2^{-1/2})\eps^2
   - \sigma^v_{\alpha \alpha}\int_{2^{-1/2}\eps}^{\eps}(2x_1-2^{1/2}\eps) 
   \e^{-x_1^2/t} \dd x_1\\
   & -\sigma^v_{\alpha
     \alpha}\int_{2^{-1/2}\eps}^{\eps}\int_{2^{1/2}\eps-x_1}^{x_1}\e^{-x_2^2/t}
   \dd x_2\dd x_1
   -(\sigma^v_{\alpha \alpha})^2\int_{2^{-1/2}\eps}^{\eps}
   \int_{2^{1/2}\eps-x_1}^{x_1}\e^{-(x_1^2+x_2^2)/t} \dd x_2\dd x_1\\
   &-\int_{2^{-1/2}\eps}^{\eps}\int_{2^{1/2}\eps-x_1}^{x_1}
   \e^{-(x_1-x_2)^2/2t} \dd x_2\dd x_1
   -2\sigma^v_{\alpha \alpha}\int_{2^{-1/2}\eps}^{\eps}
   \int_{2^{1/2}\eps-x_1}^{x_1}\e^{-(x_1^2+x_2^2)/2t} \dd x_2\dd x_1\\
   &-(\sigma^v_{\alpha
     \alpha})^2\int_{2^{-1/2}\eps}^{\eps}\int_{2^{1/2}\eps-x_1}^{x_1}\e^{-(x_1+x_2)^2/2t}
   \dd x_2\dd x_1
   +O(t^\infty)\\
   =& \frac{1}{2}\eps^2 + \sigma^v_{\alpha \alpha}\frac{t}{2}
   +\sigma^v_{\alpha \alpha}(-\frac{1}{2}t + 
   \frac{\sqrt{\pi}}{2}\eps t^{1/2}) + (\sigma^v_{\alpha \alpha})^2\frac{1}{8}\pi t\\
   & - t + 2^{-1/2}\sqrt{\pi}\eps t^{1/2}
   +2\sigma^v_{\alpha \alpha}\frac{1}{4}\pi t + 
   (\sigma^v_{\alpha \alpha})^2\frac{1}{2}t \\
   &- (\frac{3}{8}-2^{-1/2})\eps^2  - 0 -0 - 0 -
   2^{-1/2}\sqrt{\pi}(1-2^{-1/2})\eps t^{1/2}+\frac{t}{2}\\
   & - 0 - 0 + O(t^\infty) \\
   =& \left(\frac{1}{2}- \frac{3}{8}+2^{-1/2}\right)\eps^2 +
   \sigma^v_{\alpha \alpha}\frac{\sqrt{\pi}}{2}\eps t^{1/2} +
   \frac{\sqrt{\pi}}{2}\eps t^{1/2}\\
   &+ (\sigma^v_{\alpha \alpha})^2\frac{1}{8}\pi t - t +
   \sigma^v_{\alpha \alpha}\frac{1}{2}\pi t +(\sigma^v_{\alpha
     \alpha})^2\frac{1}{2}t +\frac{t}{2} + O(t^\infty)
 \end{align*}
 The various terms that give zero contribution can all be bounded by
 noticing that the integrated function can be bounded as $\e^{-C/t}$
 for some constant $C>0$ over the entire domain of integration.

 Let $\Omega$ denote the region we integrated over, then this gives
 the following contribution
\begin{align*}
  & \int_\Omega e_{E}(t,(x_1,x_2),(x_1,x_2)) \dd\vol \\
  =& \frac{1}{4\pi t}\vol(\Omega)
  +  \frac{1}{8\sqrt{\pi t}}\left( \sigma^v_{\alpha \alpha}+1\right)\eps \\
  & -\frac{1}{8\pi} + \frac{1}{8}\sigma^v_{\alpha \alpha} +
  \left(\frac{1}{32}+\frac{1}{8\pi}\right) (\sigma^v_{\alpha
    \alpha})^2 + O(t^\infty)
\end{align*}
so for the entire region $E$ we get
\begin{align*}
  & \int_{E}e_{E}(t,(x_1,x_2),(x_1,x_2)) \dd\vol \\
  =& \frac{1}{4\pi t}\vol(E) + \frac{1}{8\sqrt{\pi
      t}}\sum_{\alpha}\left( \sigma^v_{\alpha \alpha}+1\right)\eps +
  \frac{1}{8\sqrt{\pi t}}\sum_{\alpha < \beta}
  (\sigma^v_{\alpha \alpha}+\sigma^v_{\beta \beta})\eps \\
  & -\frac{\deg(v)}{8\pi} + \frac{1}{8}\sum_{\alpha}\sigma^v_{\alpha
    \alpha} + \left(\frac{1}{32}+\frac{1}{8\pi}\right)\sum_{\alpha}
  (\sigma^v_{\alpha \alpha})^2 \\
  & + \frac{1}{16}\sum_{\alpha < \beta}\sigma^v_{\alpha \alpha}
  \sigma^v_{\beta \beta}
  +\frac{1}{4\pi}\sum_{\alpha < \beta}(\sigma^v_{\alpha \beta})^2  + O(t^\infty)\\
\end{align*}
 
After doing a similar computation for the pieces of type $A$, $B$ and
$D$ we get the following heat asymptotics.
 
\begin{theorem}
  \label{thm:main-example}
  Let $M:= G \times G / ((x,y)\sim(y,x))$ where $G=(V,E)$ is a metric
  graph with the standard Kirchhoff boundary conditions at all
  vertices. Then the heat kernel asymptotics of $M$ are
 \begin{align*}
   & \int_{M} p(t,(x_1,x_2),(x_1,x_2)) \dd x_1\dd x_2 \\
   \to_{t \to 0^+}& \frac{1}{4\pi t}\vol(M) 
   + \frac{1}{8\sqrt{\pi t}}\left(2L(G)(|V|-|E|) +  \sqrt{2}L(G) \right) \\
   &+\frac{1}{16} \sum_{v \neq v'}(2-\deg(v))(2-\deg(v'))
   + \sum_{v \in V}\left(\frac{3}{8}-\frac{\deg(v)}{4}
     +\frac{\deg(v)^2}{32} \right) 
   + O(t^\infty)\\
 \end{align*}
\end{theorem}
The first two terms of the asymptotic expansion are the volume and the
lengths of the boundary (the boundary here consists of the terms from
the product and the symmetrization).  The constant term describes the
corners and again has contributions from the product and the
symmetrization.

A vertex of degree $2$ in a metric graph with Kirchhoff boundary
condition imposes no conditions on the functions.  Hence this vertex
should be invisible to the heat kernel and one can easily check that
degree two vertices give no contribution to the asymptotics above.  A
vertex of degree one will produce a wedge with opening angle
$\frac{\pi}{4}$ in $M$.  This can be compared to the known heat
asymptotics of planar polygons~\cite{BergSrisatkunarajah88}. The
constant term contribution matches the expected $\frac{5}{32}$.

%
%

\ifthenelse{\isundefined \OlafsVersion}
{  
  \bibliographystyle{amsalpha}    
  \bibliography{literatur2}}
{   
  \bibliographystyle{/home/post/Aktuell/BibTeX/my-amsalpha}
  \bibliography{/home/post/Aktuell/BibTeX/literatur}
}

%
\appendix
\section{Measure theoretic background}
\label{app:measure}
%

\begin{definition}
  A function $(t,x,U) \mapsto p_t(x,U)$ with $t \in (0,\infty)$, $x
  \in M$ and $U \in \mathcal{B}(M)$ is called a \emph{$\mu$-symmetric
    Markovian transition function} if it satisfies the following
  \begin{enumerate}
  \item Measurability: for fixed $t$ and $x$, the function $p_t(x,
    \cdot)$ is a positive measure and for fixed $t$ and $U$, the
    function $p_t(\cdot, U)$ is $\mathcal{B}(M)$-measurable.
  \item Semigroup property: $p_tp_s f = p_{t+s}f$ for all $f \in
    \mathcal{B}_b(M)$
  \item Markov property: $p_t(x,M) = 1$
  \item $\mu$-symmetry: $\int_M f(x) (p_t g)(x) \dd\mu(x) = \int_M
    (p_tf)(x) g(x) \dd\mu(x)$ for any $f,g$ non-negative measurable
    functions
  \item Unit mass: $\lim_{t \to 0}p_tf (x) = f(x)$
  \end{enumerate}
  Here $(p_tf)$ is defined as $(p_tf)(x):=\int_M f(y) p_t(x,dy)$ where
  $p_t(x,dy)$ means we are integrating with respect to the measure
  $p_t(x, \cdot)$.
\end{definition}

The following proposition says that heat kernels and transition
functions are equivalent in our setting. The absolute continuity condition is
satisfied by \Rem{absolute_continuity}.

\begin{proposition}
  For $x$ and $t>0$ fixed, the probability measure $p_t(x,U)$ admits a
  probability density function $p_t(x,y)$ such that
  \begin{align*}
    p_t(x,U) = \int_U p_t(x,y)\dd\mu(y)
  \end{align*}
  for all measurable $U \in \mathcal{B}(M)$ if and only if the measure
  $p_t(x, \cdot)$ is absolutely continuous with respect to
  $\mu(\cdot)$, that is $p_t(x,U)=0$ for any set $U$ with $\mu(U)=0$.
\end{proposition}
\begin{proof}
  As $M$ is a compact metric space, $\mu$ is
  $\sigma$-finite. Hence, this is the exact statement of the
  Radon-Nikod\'ym theorem, the density function is the Radon-Nikod\'ym
  derivative.
\end{proof}

\begin{definition}
  Let $(\Omega, \mathcal{F}, \Proba_\Omega)$ be a probability space
  ($\Omega$ a set, $\mathcal{F}$ a $\sigma$-algebra and
  $\Proba_\Omega$ a probability measure on $\mathcal F$).  Let
  $(M,\mathcal B,\mu)$ be a measure space and let $I$ be a subset of
  $[0,\infty)$.  A family $\{X_t\}_{t\in I}$ of measurable maps
  $\map{X_t}{(\Omega, \mathcal{F})} {(M,\mathcal B)}$ is called a
  \emph{stochastic process} on $\Omega$ with values in $M$ and index
  set $I$.
\end{definition}

\begin{definition}
  For a stochastic process $X_t$, let
  \begin{align*}
    \mathcal{F}_t^0 := \sigma\bigl(X_s, s \in [0,t]\bigr),
  \end{align*}
  where $\sigma(\cdot)$ denotes the smallest $\sigma$-algebra
  containing all sets in the brackets and we use the convention that
  $X_s$ understood as a $\sigma$-algebra means the $\sigma$-algebra $
  X_s^{-1}(\mathcal{B})$.

\end{definition}
Note that $\mathcal{F}_s^0 \subseteq \mathcal{F}_t^0 \subseteq
\mathcal{F}$ for any $s <t$ by definition. Hence this is an increasing
sequence of $\sigma$-algebras.

\begin{definition}[\cite{StroockVaradhan79}]
  \label{def:cond-proba}
  Let $\mathcal{G} \subset \mathcal{F}$ by a sub-$\sigma$-algebra,
  then the \emph{conditional probability distribution} $Q_\omega$ of
  $\Proba_\Omega$ given $\mathcal{G}$ is a family of probability
  measures on $(\Omega, \mathcal{F})$ indexed by $\omega \in \Omega$
  such that
  \begin{itemize}
  \item for each $B \in \mathcal{F}$, the function $Q_\omega(B)$ is
    $\mathcal{G}$-measurable as a function of $\omega$
  \item for $A \in \mathcal{G}$ and $B \in \mathcal{F}$ we have 
    \begin{equation*}
      \Proba_\Omega(A \cap B) = \int_A Q_\omega(B) \dd \Proba_\Omega(\omega).
    \end{equation*}
  \end{itemize}
  A conditional probability distribution is unique in the sense that
  two choices agree $\Proba_\Omega$-almost surely. If $\Omega$ is a
  Polish space, a conditional probability distribution always exists.
\end{definition}

\begin{remark}
  If $\mathcal{G}=\mathcal{F}$, then $Q_\omega(B) = \chi_B(\omega)$
  that is $Q_\omega(B)$ is the indicator function of the set $B$. When
  $\mathcal{G}$ is a proper sub-$\sigma$-algebra, one can picture the
  $Q_\omega$ as $\mathcal{G}$-measurable approximations of the
  indicator function.
\end{remark}

\begin{remark}
  One writes $\Proba_\Omega( X_t \in U | \mathcal{F}_t^0)$ with $U \in
  \mathcal{B}$ as a shorthand for the map from $\Omega$ to $\R$
  that maps $\omega \mapsto Q_\omega(X_t \in U)$ where $Q_\omega$ is
  the conditional probability distribution of $\Proba_\Omega$ given
  $\mathcal{F}_t^0$.
\end{remark}

\begin{definition}
  Let $(M,d,\mu)$ be a metric measure space with Borel
  $\sigma$-algebra $\mathcal B(M)$.  A \emph{continuous normal Markov
    process} on the set of continuous paths consists of four elements
  \begin{enumerate}
  \item $\Omega:=\mathcal{P}(M)$ is the set of continuous paths $\map
    \omega {[0,\infty)} M$;
  \item $\mathcal{B}(\mathcal{P}(M))$ is the $\sigma$-algebra of Borel
    sets on it;
  \item $\{\Proba_x\}_{x \in M}$ are probability measures on
    $\mathcal{P}(M)$ that satisfy $\Proba_x\left( \omega(0)=x \right)
    =1$;
  \item $X_t(\omega):=\omega(t)$ is a stochastic process on
    $\mathcal{P}(M)$ with values in $M$
  \end{enumerate}
  satisfying the \emph{Markov property}
  \begin{align*}
    \Proba_x(X_{t+s} \in U | \mathcal{F}_t^0) = \Proba_{X_t}(X_s \in U)
    && \text{$\Proba_x$-almost surely}
  \end{align*}
  for all $U \in \mathcal{B}(M)$ and all $s,t \ge 0$.
\end{definition}

By convention one often refers to $X_t$ as the Markov process. The
existence of the other objects is then implicitly assumed.  
We will also write
$\Proba$ to denote the family of Borel probability measures
$\Proba_x$. We will refer to the measures $\Proba_x$ as Wiener
measures.

We have
\begin{align*}
  \mathbb{E}_x[f(X_t)]= \int_{\mathcal{P}_x(M)} f(\omega(t))
  d\mathbb{P}_x(\omega) = \int_M f(y) p_t(x,y) d\mu(y)
\end{align*}
for a continuous function $f$ on $M$.

\begin{definition}
  \label{def:strong-markov}
  A stochastic process satisfies the \emph{strong Markov property} if
  \begin{align*}
    \mathbb{P}_x(X_{\zeta+s} \in U | \mathcal{F}_{\zeta}) =
    \mathbb{P}_{X_{\zeta}}(X_s \in U)
  \end{align*}
  holds for all $s \ge 0$, $U \in \mathcal{B}(M)$ and all stopping
  times $\zeta$ (see \Def{stopping-time}). In this case the Markov
  process is called a \emph{continuous Hunt process} or a
  \emph{diffusion}.
\end{definition}
Note that the strong Markov property also implies
\begin{align*}
  \mathbb{E}^{\mathcal{F}_{\zeta}}[f(X_{s+\zeta})] =
  \mathbb{E}^{\zeta}[f(X_s)].
\end{align*}

\end{document}